\documentclass[a4paper,11pt]{amsart}

\usepackage[english]{babel}
\usepackage[centertags]{amsmath}
\usepackage{latexsym, amsfonts, amssymb, graphicx}
\usepackage[all,dvips]{xy}
\usepackage[latin1]{inputenc}
\usepackage{url}

\newtheorem{thm}{Theorem}[section]
\newtheorem{lem}[thm]{Lemma}
\newtheorem{prop}[thm]{Proposition}
\newtheorem{cor}[thm]{Corollary}
\newtheorem{defn}[thm]{Definition}

\newtheorem{question}[thm]{Question}

\newtheorem{rem}[thm]{Remark}
\newtheorem{exmp}[thm]{Example}

\def\N{{\mathbf N}}

\def\Z{{\mathbf Z}}
\def\P{{\mathbf P}}
\def\C{{\mathbf C}}

\newcommand{\OP}{\mathcal{O}_{S,P}}
\newcommand{\OC}{\mathcal{O}_{S,C}}

\newcommand{\mC}{\mathfrak{m}_{S,C}}

\newcommand{\hOC}{\widehat{\mathcal{O}}_{S,C}}

\newcommand{\p}{\partial}
\newcommand{\w}{\wedge}
\newcommand{\F}{\mathbf{F}}
\newcommand{\res}{\textrm{res}}
\newcommand{\supp}{\textrm{Supp}}

\def\fract#1/#2{\hbox{\leavevmode
\kern.1em \raise .5ex \hbox{\the\scriptfont0 $#1$}\kern-.1em }/
\hbox{\kern-.15em \lower .25ex \hbox{\the\scriptfont0 $#2$}}
}

\begin{document}

\title[Residues on surfaces and application to coding theory]{Sums of residues on algebraic surfaces and application to coding theory}
\author{Alain \textsc{Couvreur}}

\address{Institut de Math\'ematiques de Toulouse, UMR 5219, Universit\'e Paul Sabatier, 118 route de Narbonne, 31062 Toulouse, France}
\email{couvreur@math.univ-toulouse.fr}

\begin{abstract}
In this paper, we study residues of differential $2$-forms on a smooth algebraic surface over an arbitrary field and give several statements about sums of residues.
Afterwards, using these results we construct algebraic-geometric codes which are an extension to surfaces of the well-known differential codes on curves.
We also study some properties of these codes and extend to them some known properties for codes on curves.
\end{abstract}

\maketitle



\section*{Introduction}
The present paper is divided in two parts. The first one is a theoretical study of residues of differential $2$-forms on algebraic surfaces over an arbitrary field.
The second one uses results of the first part to construct differential codes on algebraic surfaces and to study some of their properties. The reader especially interested in coding theory is encouraged to read briefly the definitions and the results of the first part and then to jump to the second part.

\subsection*{About residues}
If the notion of residue is well-known for differential forms on curves, there is no unified definition in higher dimension.
On complex varieties, one can distinguish two objects called residues in the literature. The first one appears for instance in Griffiths and Harris \cite{gh} chapter V.
In this book, given an $n$-dimensional variety $X$, the residue of a meromorphic $n$-form $\omega$ at a point $P$ is a complex number obtained by computing an integral on a real $n$-cycle. 
This object depends on some $n$-uplet of divisors whose sum contains the poles of $\omega$ in a neighborhood of $P$.
Another definition is given in \textit{Compact complex surfaces} by Barth, Hulek, Peters and Van De Ven (\cite{bhpv} II.4). In this book, given an $n$-dimensional variety $X$ and a one-codimensional subvariety $Y$ of $X$,
the residue along $Y$ of a $q$-form on $X$ having a simple pole along $Y$ is a $(q-1)$-form on $Y$. The computation of this residue can be obtained by a combinatorial way, or by computing an integral on a real subvariety (\cite{bhpv} II.4). 

In algebraic geometry over an arbitrary field, several references deal with \textit{residues}, for instance Hartshorne \cite{harRD} or Lipman \cite{lip}.
Actually their main objective is to establish duality theorems generalizing Serre's one. Thus, their first intention is not to define residues of differential forms on higher-dimensional varieties over an arbitrary field.

The goal of the first part of this paper is to generalize to surfaces over an arbitrary field, the definitions of residues given
for complex varieties  in \cite{bhpv} and in \cite{gh}.
Then, we will establish results of independence on the choice of local coordinates, and focus on summation properties. 
Notice that Hartshorne, in \cite{harRD} III.9, introduces a \textit{Grothendieck residue symbol} having slightly the same properties as the residue defined in Griffiths and Harris's book. Moreover, Lipman in \cite{lip} section 12 states a summation residue formula which is closed to the theorem \ref{RF3} in the present paper.
Finally, most of the results of this first part can be considered as consequences of several statements lying in \cite{harRD} or \cite{lip}. Nevertheless, both references are long and contain an important functorial machinery which is not necessary to define residues on surfaces, to study their properties and to obtain summation formulas. That is why we decided to present a self-contained paper for which references \cite{harRD} and \cite{lip} are not prerequisites.

Notice that we chose to work only on surfaces.
At least two main reasons justify this choice.
First, working on $n$-dimensional varieties would have given too heavy notations.
Second, the most difficult step in algebraic-geometric coding theory lies between curves and higher-dimensional varieties

\subsection*{About algebraic-geometric codes}
In coding theory, two main problems are frequently studied. The first one is: \textit{how to find a lower bound for its minimal distance of a given code?} The second one is: \textit{how to find algorithms correcting a suitable number of errors in a \textit{reasonable} time?} 
Given an arbitrary code defined by a generator or a parity-check matrix, both problems are very difficult. A good way to solve them, is to get a geometric (or arithmetic) \textit{realization} of the code.
Then, one or both problems may be translated into geometric (or arithmetic) problems. This is, for instance, successfull for the study of Reed-Muller codes.
Consequently, geometric constructions of codes are often interesting.

\subsubsection*{Codes on curves} In 1981, Goppa introduced in a construction of error-correcting codes using algebraic curves (see \cite{goppa}). Their study has been a fruitful topic of research during last thirty years.
Hundreds of papers are devoted to this subject.
One of the main reasons why these codes have been so intensively studied is that some families of such codes have excellent asymptotic parameters. Particularly, Tsfasman, Vl{\u{a}}du\c{t} and Zink proved in \cite{TVZ} that some families of algebraic-geometric codes beat the Gilbert-Varshamov bound.
Most of the basic results about codes on curves are summarized in \cite{step},  \cite{sti} chapter II and \cite{TV}.

\subsubsection*{Codes on higher-dimensional varieties} In higher dimension, the topic has not been as extensively explored.
The first general construction of algebraic-geometric code from a variety of arbitrary dimension has been given by Manin in the paper with Vl{\u{a}}du\c{t} \cite{man}.
Afterwards, codes coming from some particular varieties have been studied. Among others, in \cite{aubry}, Aubry
dealt with codes on quadric varieties.
His results have been improved in dimension $2$ and $3$ by Edoukou in \cite{fred2} and \cite{fred3}.
Codes on Grassmannians have been discussed by Nogin in \cite{nogin} then by Ghorpade and Lachaud in \cite{gl}.
Codes on Hermitian varieties have been treated by Chakravarti in \cite{chak2}, then by Hirschfeld, Tsfasman and Vl{\u{a}}du\c{t} in \cite{HTV}, afterwards by S\o rensen in his PhD thesis \cite{sorensen} and by Edoukou in \cite{fred}.
In \cite{rodier}, Rodier presented a unified point of view for all the above-cited examples regarding these varieties as flag-varieties and gave some more examples of codes.
Zarzar studied in \cite{zarzar} the parameters of codes on surfaces having a small Picard number. The author proposed also a decoding algorithm for such codes in a join work with Voloch \cite{agctvoloch}.
General bounds on the parameters of codes on algebraic varieties of arbitrary dimension have been given by Lachaud in \cite{lachaud} and by S\o ren Have Hansen in \cite{soH}. 
Finally, a survey paper \cite{little} by Little summarizing most of the known works on codes on higher-dimensional varieties appeared recently. 

 Notice that almost all the references cited below, deal with the question of bounding or evaluating the parameters of some error-correcting codes. This will not be the purpose of the present paper whose objective is to give general theoretical statements extending some known results for codes on curves.

\subsubsection*{Different construction of codes on curves} In the theory of algebraic-geometric codes on curves, one can distinguish two different constructions.
\textit{Functional codes} are obtained by evaluating elements of a Riemann-Roch space at some set of rational points on a curve.
\textit{Differential codes} are obtained by evaluating residues of some rational differential forms at these points.
For higher-dimensional varieties,
only the functional construction has been extended and studied (see references below).
The differential one does not seem to have a natural generalization and this question has never been treated before.

\subsubsection*{Motivations}
There are at least three motivations for an extension to surfaces of the differential construction.
The first one is historical. Indeed, the first construction of algebraic-geometric codes given by Goppa in \cite{goppa} used differentials. This construction generalized that of classical Goppa codes which can be regarded as differential codes on the projective line.
The second one is that the orthogonal of a functional code on a curve is a differential one.
Moreover, this statement is used in almost all known algebraic decoding algorithms (see \cite{hohpel}).
The third motivation is that, as said before, it is always interesting to have a geometric realization of a code.
To finish with motivations, notice that the introduction of the above cited survey paper \cite{little} of Little contains the following sentences.

\smallbreak

\textit{``In a sense, the first major difference between higher dimensional varieties and curves is that points on $X$ of dimension $\geq 2$ are subvarieties of codimension $\geq 2$, not divisors. This means that many familiar tools used for Goppa codes (e.g. Riemann-Roch theorems, the theory of differentials and residues etc.) do not apply exactly in the same way.''}

\smallbreak

Thus, finding another way of applying residues and differentials for codes on surfaces must be interesting.
This is the purpose of the second part of this paper, which starts with the presentation of a construction of codes using residues of differential $2$-forms on surfaces.
Then, connections between these codes and the functional ones are studied. We proves that any differential code is included in the orthogonal of a functional one but that the reverse inclusion is false, which is an important difference with the theory of codes on curves.
Notice that Voloch and Zarzar suggested the existence of such a difference in \cite{agctvoloch} section 3 without proving it.
Finally, we prove that, as for codes on curves, a differential code can always be regarded as a functional one associated with some parameters depending on a canonical divisor.  

\subsection*{Contents}
The first part contains sections \ref{onentwo} to \ref{RF}.
In section \ref{onentwo}, we recall the definition of one-codimensional residues along a curve $C$ of a differential $2$-form $\omega$ having $C$ as a simple pole.
Then, we define naturally the two-codimensional residue of $\omega$ along $C$ at a smooth point $P\in C$ to be the residue at $P$ of the one-codimensional residue.
In section \ref{lolo}, we study Laurent series expansions in two variables, in order to have a combinatorial definition for residues, which will be more convenient for computations.
In section \ref{general}, we introduce new definitions of one- and two-codimensional residues holding for any rational $2$-form.
Then, we prove that the two-codimensional residue at a point $P$ along a curve $C\ni P$ of a rational $2$-form does not depend on the choice of local coordinates.
In section \ref{propres}, we study some properties of one- and two-codimensional residues.
In section \ref{secsing}, we define two-codimensional residues along a curve at a singular point of it.
Finally, section \ref{RF} contains three statements about summations of residues.

The second part contains sections \ref{codesoncurves} to \ref{examples}.
Section \ref{codesoncurves} is a quick review on the theory of codes on curves.
In section \ref{codesonsurfaces}, after a brief overview on functional codes on higher-dimensional varieties, we define differential codes on surfaces.
Then, properties of these codes and their relations with functional ones are studied in section \ref{propcodes}. Particularly, we prove that a differential code is contained in the orthogonal of a functional one.
Finally, section \ref{examples} proves that the reverse inclusion may be false by treating the elementary example of the surface $\P^1 \times \P^1$.

\part{Residues of a rational $2$-form on a smooth surface}\label{partieres}

\section*{Notations}

For any irreducible variety $X$ over a field $k$, we denote by $k(X)$ its function field.
If $Y$ is a closed irreducible subvariety of $X$, then the local ring (resp. its maximal ideal) of regular functions \textit{in a neighborhood of $Y$},  that is functions which are regular in at least one point of $Y$, is denoted by 
$\mathcal{O}_{X,Y}$ (resp. $\mathfrak{m}_{X,Y}$).
The $\mathfrak{m}_{X,Y}$-adic completion of the ring $\mathcal{O}_{X,Y}$ is denoted by $\widehat{\mathcal{O}}_{X,Y}$ and its maximal ideal $\mathfrak{m}_{X,Y}\widehat{\mathcal{O}}_{X,Y}$ by $\widehat{\mathfrak{m}}_{X,Y}$.
For any function $u\in \mathcal{O}_{X,Y}$, we denote by $\bar{u}$ its restriction to $Y$.
Recall that, if $Y$ has codimension one in $X$ and is not contained in the singular locus of $X$, then $\mathcal{O}_{X,Y}$ is a discrete valuation ring with residue field $k(Y)$. In this situation, the valuation along $Y$ is denoted by $val_Y$.
Finally, we denote by $\Omega^i_{k(X)/k}$ the space of $k$-rational differential $i$-forms on $X$.


\section{One and two-codimensional residues}\label{onentwo}

\subsection*{Context}
In this section, $k$ denotes an arbitrary field of arbitrary characteristic and $S$ a smooth geometrically integral quasi-projective surface over $k$. Moreover, $C$ denotes an irreducible geometrically reduced curve embedded in $S$ and $P$ a smooth rational point of $C$.
  
\subsection{First definitions for residues}
Given a $2$-form $\omega \in \Omega^2_{k(S)/k}$, one can construct two objects called \textit{residues} in the literature.
The first one is a rational $1$-form on a curve embedded in $S$ and the second one is an element of $k$ (or of some finite extension of it). Their definitions will be the respective purposes of definitions \ref{dinv1res} and \ref{2res}.
We first need next proposition, asserting the well-definition of one-condimensional residues (definition \ref{dinv1res}).

\begin{prop}\label{inv1res}
Let $v$ be a uniformizing parameter of $\OC$ and
 $\omega$ be a rational $2$-form on $S$ having $\mC$-valuation greater than or equal to $-1$. Then, there exists $\eta_1 \in \Omega^1_{k(S)/k}$ and $\eta_2 \in \Omega^2_{k(S)/k}$, both regular in a neighborhood of $C$ and such that
\begin{equation}\label{decdec}
\omega=\eta_1\wedge \frac{\displaystyle dv}{\displaystyle v}+\eta_2.
\end{equation}

\noindent Moreover, the differential form  ${\eta_1}_{|C}\in \Omega_{k(C)/k}^1$ is unique and depends neither on the choice of $v$ nor on that of the decomposition (\ref{decdec}).
\end{prop}

\begin{proof} We first prove the existence of a decomposition (\ref{decdec}).
Recall that
\begin{equation}\label{dimform}
\dim_{k(S)}\Omega^1_{k(S)/k}=2\quad \textrm{and}\quad
\dim_{k(S)}\Omega^2_{k(S)/k}=1,
\end{equation}
(see \cite{sch1} thm III.5.4.3).
Consequently, there exists a rational $1$-form $\mu$, which is non-$k(S)$-colinear with $\frac{dv}{v}$. Thus, $\mu\w \frac{dv}{v}\neq 0$. From (\ref{dimform}), there exists also a unique function $f\in k(S)$ satisfying
$$
\omega=f\mu\w \frac{dv}{v}.
$$
 
\noindent Since $val_C(\omega)\geq -1$, the $1$-form $f\mu$ has no pole along $C$. We obtain a decomposition (\ref{decdec}) by setting $\eta_1:=f\mu$ and $\eta_2:=0$.

Obviously, this decomposition is far from being unique. Only ${\eta_1}_{|C}$ is unique.
To prove uniqueness and independence of ${\eta_1}_{|C}$ under the choice of $v$, see \cite{bhpv} II.4.
Even if this book only deals with complex surfaces, the very same proof holds for surfaces over an arbitrary field.
\end{proof}

\begin{rem}
  Another proof of proposition \ref{inv1res} will be given in section \ref{general} in a more general context (see lemma \ref{polesimple}).
\end{rem}

\begin{defn}\label{dinv1res}
  Under the assumptions of proposition \ref{inv1res} and given a decomposition of the form (\ref{decdec}) for $\omega$, the $1$-form ${\eta_1}_{|C}\in \Omega^1_{k(C)/k}$ is called the one-codimensional residue (or the $1$-residue) of $\omega$ along $C$ and denoted by
$$
\res^1_C (\omega):={\eta_1}_{|C}.
$$ 
\end{defn}

\begin{defn}\label{2res}
Under the assumptions of proposition \ref{inv1res},
let $P$ be a $k$-rational point of $C$. The two-codimensional residue (or the $2$-residue) of $\omega$ at $P$ along $C$ is the residue at $P$ of the $1$-residue of $\omega$ along $C$. That is
$$
\res^2_{C,P}(\omega):= \res_P (\res^1_{C,P} (\omega)).
$$  
\end{defn}

Notice that to define residues in this way, $\omega$ needs to have valuation greater than or equal to $-1$ along $C$.
However, two-codimensional residues can actually be defined for any rational differential form even if it has a multiple pole along $C$.
This will be the purpose of sections \ref{lolo} to \ref{propres}.

\begin{rem}
It would have been natural to define $2$-residues at a closed point $P$ of $C$. Nevertheless, we decided to keep a more geometric point of view, even if the base field is not supposed to be algebraically closed.
Notice that any geometric point of $S$ (i.e. a closed point of $S\times_k \bar{k}$) is a rational point of $S\times_k L$ for a suitable finite scalar extension $L/k$. Consequently, if we define residues at rational points of $S$, it is easy to extend this definition to geometric points using such a scalar extension.
The only arithmetic statement we will need in the second part of the present paper is that, if $C$ is defined over $k$ and $P\in C(k)$, then the $2$-residue along $C$ at $P$ of a $k$-rational $2$-form is in $k$.
That is why we keep considering non-algebraically closed fields in sections \ref{onentwo} to \ref{general} and \ref{secsing}.

However, in sections \ref{propres} and \ref{RF}, when we deal with properties of residues and particularly with summations of them, we work over an algebraically closed field.
\end{rem}

\section{Laurent series in two variables}\label{lolo}

As is well-known, the residue at a point $P$ on a curve $C$ of a $1$-form can be computed using Laurent series expansions.
The residue of a differential form at a point $P$ is the coefficient of degree $-1$ of its Laurent series expansion. We look for a similar definition in the two-dimensional case. For this purpose, we introduce Laurent series in two variables.  

\subsection*{Context}
The context of this section is exactly that of section \ref{onentwo} (see page \pageref{onentwo}).


\subsection{Laurent series expansion, the first construction}\label{laoukialediagramme}

Recall that, $C$ is assumed to be a geometrically reduced irreducible curve over $k$ embedded in $S$ and $P$ a smooth rational point of $C$.

\begin{defn}
  A pair $(u,v) \in \mathcal{O}_{S,P}^2$ is said to be a strong $(P,C)$-pair if the following conditions are satisfied.
\begin{enumerate}
\item $(u,v)$ is a system of local parameters at $P$.
\item $v$ is a uniformizing parameter of $\OC$.
\end{enumerate}
\end{defn}

 
\begin{lem}\label{kuv}
Let $(u,v)$ be a strong $(P,C)$-pair, then
there exists a morphism $\phi: k(S) \hookrightarrow k((u))((v))$ sending
$\OP$ into $k[[u,v]]$ and
$\OC$ into $k((u))[[v]]$.
\end{lem}

\begin{proof}
We will prove the existence of $\phi_0: \OC \hookrightarrow k((u))[[v]]$ entailing that of $\phi$, thanks to the universal property of fraction fields.
From \cite{sch1} II.2, any element of $\OP$ has a unique Taylor series expansion in the variables $u,v$.
Then, notice that $\OC$ and ${\OP}_{(v)}$ are isomorphic and consider the following diagram.
$$
\xymatrix{\relax \OP  \ar@{^{(}->}[d] \ar[r]_{\textrm{loc}} &
\OC \ar[r]_{\textrm{comp}} \ar@{^{(}->}[d]_{\exists !}
& \hOC \ar[d]_{\exists !} \\
 k[[u,v]] \ar[r]^{\textrm{loc}} 
 & k[[u,v]]_{(v)} \ar[r]^{\textrm{comp}}
 & \widehat{k[[u,v]]}_{(v)}.
}
$$

\noindent The horizontal arrows in the left hand square correspond to localizations, the ones in the right hand square correspond to $(v)$-adic completions.
Vertical arrows are obtained by applying respectively universal properties of localization and completion.
We now have to prove that $\widehat{k[[u,v]]}_{(v)}$ is isomorphic to $k((u))[[v]]$, which is a consequence of Cohen's structure theorem (see \cite{eis} thm 7.7 or \cite{cohen} thm 9 for an historical reference).

  
\end{proof}



\subsection{Laurent series, the second construction}\label{second}
Let $(u,v)$ be a strong $(P,C)$-pair.
Cohen's structure theorem asserts that $\hOC$ is isomorphic to $k(C)[[v]]$.
Unfortunately, this isomorphism is not always unique. Indeed, \cite{cohen} thm 10(c) asserts that, if $\textrm{Char}(k)>0$, then there are infinitely many subfields of $\hOC$ which are isomorphic to the residue field $k(C)$.
Therefore, to use this isomorphism for Laurent series expansions, we have to choose a representant of $k(C)$ which is, in some sense, related to $u$.

\begin{prop}[The field $\mathcal{K}_u$]\label{Ku}
Let $u\in \OC$ whose restriction $\bar{u}$ to $C$ is a separating element (see \cite{sti} p. 127 for a definition) of $k(C)/k$. Then, there exists a unique subfield $\mathcal{K}_u \subset \hOC$ containing $k(u)$ and isomorphic to $k(C)$ under the morphism $\hOC \rightarrow \hOC / \widehat{\mathfrak{m}}_{S,C} $. Furthermore, this field is generated over $k(u)$ by an element $y\in \hOC$. 
\end{prop}

\begin{proof}
The extension $k(C)/k(\bar{u})$ is finite and separable. Thus, from the primitive element theorem, there exists a function $\bar{y}\in k(C)$ generating $k(C)$ over $k(\bar{u})$.  
From Hensel's lemma, $\bar{y}$ lifts to an element $y\in \hOC$ and the subring $\mathcal{K}_u:=k(u)[y]\subset \hOC$ is the expected copy of $k(C)$.
The uniqueness of $\mathcal{K}_u$ is a consequence of the uniqueness of the Hensel Lift $y$ of $\bar{y}$.
\end{proof}

\begin{cor}\label{Ku[[v]]}
  Under the assumptions of proposition \ref{Ku}, any element $f\in \hOC$ has a unique expansion in $\mathcal{K}_u[[v]]$.
\end{cor}

\begin{proof}
\textbf{Existence.} Let $f$ be an element of $\hOC$ and 
$f_0$ be the Hensel-lift in $\mathcal{K}_u$ of $f \mod\ \widehat{\mathfrak{m}}_{S,C}$. The $\widehat{\mathfrak{m}}_{S,C}$-adic valuation of $f-f_0$ is greater than or equal to one. By induction, using the same reasoning on $v^{-1}(f-f_0)$, we obtain an expansion $f=f_0+f_1v+\cdots$ for $f$.

\noindent \textbf{Uniqueness.}
Assume that $f$ has two distinct expansions $\sum_j f_j v^j$ and $\sum_j \widetilde{f}_j v^j$ in $\mathcal{K}_u[[v]]$. Let $j_0$ be the smallest integer such that $f_{j_0}\neq \widetilde{f}_{j_0}$.
From proposition \ref{Ku}, a nonzero element of $\mathcal{K}_u\subset \hOC$ has $\widehat{\mathfrak{m}}_{S,C}$-adic valuation zero. Consequently, $0$ has $\widehat{\mathfrak{m}}_{S,C}$-adic valuation $j_0$ which is absurd.
\end{proof}


The second Laurent series expansion using Cohen's structure theorem needs weaker conditions on the pair $(u,v)$. Thus, before we define it, we give a new definition.


\begin{defn}
A pair $(u,v)\in \OC^2$ is said to be a weak $(P,C)$-pair if $\bar{u}$ is a uniformizing parameter of $\mathcal{O}_{C,P}$ and  $v$ is a uniformizing parameter of $\OC$. 
\end{defn}

\begin{rem}
Obviously, a strong $(P,C)$-pair is weak, but the converse statement is false (see next example).  
\end{rem}

\begin{exmp}
Assume that $S$ is the affine plane over $\C$, the curve $C$ is the line of equation $y=0$ and $P$ is the origin. Set $u:= \frac{(x+y)(x-y)}{x}$ and $v:=xy$. Then, $(u,v)$ is a weak $(P,C)$-pair which is not strong.  
\end{exmp}

\noindent Now, we can define the second way of Laurent series expansion.

\begin{lem}\label{kuv2}
Given a weak $(P,C)$-pair $(u,v)$, there is an injection
$
\varphi:\ k(S)\hookrightarrow k((u))((v))
$
sending $\OC$ in $k((u))[[v]]$.   
\end{lem}

\begin{proof}
As in the proof of lemma \ref{kuv}, we just have to prove the existence of a morphism $\varphi_0:\OC \hookrightarrow k((u))[[v]]$.
The curve $C$ is assumed to be geometrically reduced, thus from \cite{mumford} prop II.4.4 (i), the extension $k(C)/k$ is separable, hence has a separable transcendence basis.
Moreover, the function $\bar{u}$ is a uniformizing parameter of $\mathcal{O}_{C,P}$, thus its differential $d\bar{u}$ is nonzero and, from \cite{bou} thm V.16.7.5, it is a separating element of $k(C)/k$.
From proposition \ref{Ku}, there is an injection $\OC \hookrightarrow \mathcal{K}_u[[v]]$. Furthermore,
there is a natural extension $\mathcal{K}_u\hookrightarrow k((u))$,
coming from the $(\bar{u})$-adic completion of $k(C)\cong \mathcal{K}_u$. Applying this extension coefficientwise on $\mathcal{K}_u[[v]]$, we obtain the morphism $\varphi_0$.
\end{proof}

\noindent Next proposition links both Laurent series expansions.

\begin{prop}\label{sameexp}
If $(u,v)$ is a strong $(P,C)$-pair, then Laurent series expansions of lemmas \ref{kuv} and \ref{kuv2} are the same. That is $\phi=\varphi$. 
\end{prop}


\begin{proof}

Consider again the diagram in \ref{laoukialediagramme} including the new expansion
$$
\xymatrix{\relax \OP  \ar[r] \ar[d] & 
\OC \ar[r] \ar[d] \ar@/^2pc/[rrr]^{\varphi_0} \ar[rrd] |!{[r];[dr]}\hole ^(.65){\phi_0\  }   & 
\hOC \ar[d] \ar[r]^{\sim}& \mathcal{K}_u[[v]] \ar[d]^{\gamma} \ar[r]^-{\delta} &
k((u))[[v]] \ar@{.>}[ld]_{\textrm{id}}\\
k[[u,v]] \ar[r] & k[[u,v]]_{(v)} \ar[r]
& \widehat{k[[u,v]]}_{(v)} \ar[r]^{\sim} & k((u))[[v]]
.}
$$ 
Maps $\gamma$ and $\delta$ correspond respectively to the first and the second expansion.
We have to prove that $\phi_0=\varphi_0$, which is equivalent with $\gamma=\delta$. 

Recall that, from proposition \ref{Ku}, the field $\mathcal{K}_u$ is generated over $k(u)$ by an element $y\in \hOC$.
Thus, a local morphism $\mathcal{K}_u[[v]] \rightarrow k((u))[[v]]$ is entirely determined by the images of $u,v$ and $y$.
Obviously, $\delta$ sends $u$ and $v$ respectively on themselves and from the commutativity of the left part of the diagram, so does $\gamma$. 
The only nonobvious part is to prove that $\gamma$ sends $y$ on $\psi(u)$, where $\psi(\bar{u})$ is the $(\bar{u})$-adic expansion of $\bar{y}$.

Let $F\in k(\bar{u})[T]$ be the minimal polynomial of $\bar{y}$ over $k(\bar{u})$.
The formal function $y$ is the unique root of $F$ in $\hOC$ whose class in the residue field $k(C)$ is $\bar{y}$. Therefore, the morphism $\gamma$ must send $y$ on the unique root of $F$ in $k((u))[[v]]$ which is congruent to $\psi(u)$ modulo $(v)$.
Moreover, $\psi(\bar{u})=\bar{y}$, then $F(\bar{u},\psi(\bar{u}))=0$, thus the formal series $F(X,\psi(X))\in k[[X]]$ is zero. Consequently, $F(u,\psi(u))$ is zero in $k((u))$, hence is zero in $k((u))[[v]]$.
Then, $\psi(u)$ is a root of $F(u,T)\in k((u))[[v]][T]$ whose class in the residue field $k((u))$ equals $\psi(u)$, such a root is unique. Thus, $\gamma(y)=\psi(u)$.
\end{proof}

\subsection{Change of coordinates}
In this subsection, we define
$1$- and $2$-residues of any differential $2$-form $\omega$ using weak $(P,C)$-pairs. These definitions hold even if $C$ is a multiple pole of $\omega$. Afterwards, we prove that the new definition of $2$-residue does not depend on the choice of a weak $(P,C)$-pair. For that, we must describe changes of weak $(P,C)$-pairs.

\begin{lem}\label{lemcv}
Let $(u,v)$ and $(x,y)$ be two weak $(P,C)$-pairs, then the Laurent series expansions of $u$ and $v$ in $k((x))[[y]]$ are of the form
\begin{equation}\label{cv}\tag{CV}
\left\{
\begin{array}{rclcl}
u & = & f(x,y) &\ \textrm{with} & 
\ f(x,0)\in xk[[x]]\smallsetminus x^2k[[x]]\\
v & = & g(x,y) &\ \textrm{with} & 
\ g\in yk((x))[[y]]\smallsetminus y^2k((x))[[y]].     
\end{array}
\right.
\end{equation}
\end{lem}

\begin{proof}
Functions $\bar{u}$ and $\bar{x}$ are both uniformizing parameters in $\mathcal{O}_{C,P}$, thus $\bar{u}=f(\bar{x},0)\in \bar{x}k[[\bar{x}]]\smallsetminus \bar{x}^2k[[\bar{x}]]$.
Both functions $v$ and $y$ are uniformizing parameters of $\OC$, then $v/y$ is invertible in $\OC$, that is $v/y\in k((x))[[y]]^{\times}$.
\end{proof}


\section{General definition of two-codimensional residues}\label{general}

Laurent series have been introduced in section \ref{lolo} because they are useful for computations.
Using them, one can define $1$- and $2$-residues in a more general context.

\subsection*{Context}
The context of this section is exactly that of section \ref{onentwo} (see page \pageref{onentwo}).

\begin{defn}\label{bordel}
Let $\omega \in \Omega^2_{k(S)/k}$ and $(u,v)$ be a weak $(P,C)$-pair. Then, there exists an unique function $h\in k(S)$, such that $\omega=h du\w dv$ and $h$ has a Laurent series expansion $h=\sum_j h_j(u)v^j$.
\begin{enumerate}
\item The $(u,v)$-$1$-residue of $\omega$ along $C$ in a neighborhood of $P$ is defined by
$$(u,v)\res^1_{C,P}(\omega):= h_{-1}(\bar{u})d\bar{u}\ \in \Omega^1_{k(C)/k}.$$
\item The $(u,v)$-$2$-residue of $\omega$ at $P$ along $C$ is defined by
$$(u,v)\res^2_{C,P}(\omega):=\res_P((u,v)\res^1_{C,P}(\omega))=h_{-1,-1} \in k.$$
\end{enumerate}
\end{defn}

\begin{rem}\label{rational}
Proposition \ref{sameexp} asserts that $(u,v)\res^1_{C,P}(\omega)$ is a rational differential form and not a formal one.
This is the reason why we introduced this second way of Laurent series expansion.
\end{rem}

\begin{rem}
Obviously, if $val_C(\omega)\geq -1$, definition \ref{bordel} coincides with definitions \ref{dinv1res} and \ref{2res}. 
\end{rem}

\begin{rem}\label{remglob}
In this definition of one-codimensional residues, we specify the point $P$. This $1$-form is supposed to give us information about $\omega$ only in a neighborhood of $P$. 
However, we will see in section \ref{about2} that this one-codimensional residue is actually a \textit{global} object on $C$, hence independent on $P$.
\end{rem}

\noindent Now, we will prove the following statements.
\begin{enumerate}
\item One-codimensional residues do not depend on the choice of $v$.
\item Two-codimensional residues do not depend on the choice of $u$ and $v$.
\end{enumerate}

\noindent \textbf{Caution.} In what follows, we sometimes deal with formal differential forms, that is objects of the form $fdu\w dv$, where $f\in k((u))((v))$. Using such a general point of view is necessary in some parts of next proofs (for instance that of theorem \ref{inv2res} and proposition \ref{semiglob}). Definitions of one- and two-codimensional residues extend naturally to formal forms.

\begin{lem}\label{yval}
The morphism $k((u))((v))\rightarrow k((x))((y))$ given by a change of variables (\ref{cv}) in lemma \ref{lemcv} is well-defined and sends series (resp. formal forms) with $(v)$-adic valuation $n\in \Z$ on series (resp. formal forms) with $(y)$-adic valuation $n$.  
\end{lem}

\begin{proof}
See appendix \ref{apyval}.  
\end{proof}


\begin{thm}[Invariance of $2$-residues under (CV)]\label{inv2res}
Let $\omega=h(u,v)du \w dv$ be a formal $2$-form and $(x,y)\in {k((u))((v))}^2$ connected with $(u,v)$ by a change of variables of the form (\ref{cv}). Then,
$$
(u,v)\res^2_{C,P}(\omega)=(x,y)\res^2_{C,P}(\omega).
$$  
\end{thm}

\noindent The proof of this proposition will use forthcoming lemmas \ref{ypasdesouci} and \ref{polesimple}.
First, notice the change of coordinates (\ref{cv}) in lemma \ref{lemcv} can be applied in two steps. First, from $(u,v)$ to $(u,y)$, then from $(u,y)$ to $(x,y)$. That is,
$$
\begin{array}{cccc}
\textrm{first} &
\textrm{(CV1)}
\left\{\begin{array}{rcl}
u & = & u \\   
v & = & \gamma(u,y)
\end{array}\right.
,& \textrm{then} &
\textrm{(CV2)}
\left\{
\begin{array}{rcl}
u & = & f(x,y)\\
y & = & y
\end{array}
\right.
\end{array},
$$
where
$\gamma$ is a series in $yk((u))[[y]]\smallsetminus y^2k((u))[[y]]$ satisfying $g(x,y)=\gamma(f(x,y),y)$.
We will prove successively that $2$-residues are invariant under (CV1) and (CV2).

\begin{lem}[Invariance of $1$-residues under (CV1)]\label{ypasdesouci}
Let $\omega$ be a formal $2$-form. For all $y$ linked to $(u,v)$ by a change of variables (CV1): $v=\gamma(u,y)$, we have
$$
(u,v)\res^1_{C,P}(\omega)=(u,y)\res^1_{C,P}(\omega).
$$
\end{lem}

\begin{proof}
The $2$-form $\omega$ is of the form $\omega=h du\w dv$ for some $h\in k((u))((v))$. After applying (CV1), we get
$$
\omega=h(u,\gamma(u,y))\frac{\p \gamma}{\p y} du\w dy.
$$  

The field $k((u))((v))$ is the $(v)$-adic completion of the $k((u))(v)$ regarded as a function field over $k((u))$. From \cite{sti} IV.2.9, the coefficient of $v^{-1}$ in $h(u,v)$ equals that of $y^{-1}$ in $h(u,\gamma(u,y))\p \gamma/\p y$.
\end{proof}

\begin{rem}
Notice that in the whole chapter IV of \cite{sti}, the base field is assumed to be perfect, which is not true for $k((u))$ if $\textrm{Char}(k)>0$. However,the proof of IV.2.9 is purely formal and holds for non-perfect base fields. 
\end{rem}

\noindent Operation (CV2) might change $1$-residues. Nevertheless, we will see that it preserves $2$-residues.

\begin{lem}\label{polesimple}
Let $\omega$ be a formal $2$-form, $\omega=h(u,v)du \w dv$ with $h\in k((u))((v))$ such that $val_{(y)}(h)\geq -1$.
Then, for any pair $(x,y)\in k((u))((v))^2$ related to $(u,v)$ by a change of variables (\ref{cv}) of lemma \ref{lemcv}, we have
$$
(u,v)\res^1_{C,P}(\omega)=(x,y)\res^1_{C,P}(\omega).
$$    
\end{lem}

\begin{rem}
Notice that the proof of proposition \ref{inv1res} is a direct consequence of lemma \ref{polesimple}.
\end{rem}

\begin{proof}
From lemma \ref{ypasdesouci}, $(u,v)\res^1_{C,P}(\omega)=(u,y)\res^1_{C,P}(\omega)$. Thus, we only study the behavior of residues under (CV2). Decompose $\omega$ by isolating its degree $-1$ term,
$$
\omega= \frac{h_{-1}(u)}{y} du\w dy + \left( \sum_{j\geq 0}h_j(u)y^j  \right)du \w dy = \omega_{-1}+\omega_+ .
$$
The formal form $\omega_+$ has positive $(y)$-adic valuation. From lemma \ref{yval},  the change of variables (CV2) does not change this valuation.
Consequently, the $(x,y)$-$1$-residue of $\omega$ is that of $\omega_{-1}$ and
after applying (CV2), we have
$$
\omega_{-1}=\frac{h_{-1}(f(x,y))}{y}\frac{\p f}{\p x}dx\w dy.
$$
From lemma \ref{yval}, $h_{-1}(f(x,y))$ has $(y)$-adic valuation zero. Thus, 
$$
(x,y)\res^1_{C,P}(\omega)=h_{-1}(f_0(\bar{x}))f'_0(\bar{x})d\bar{x}
=h_{-1}(f_0(\bar{x}))d(f_0(\bar{x})),
$$
where $f_0(x):=f(x,0)$.
This formal $1$-form equals $(u,y)\res^1_{C,P}(\omega)=h_{-1}(\bar{u})d\bar{u}$, using the change of variables $\bar{u}=f(\bar{x},0)$.
\end {proof}

\noindent For the proof of theorem \ref{inv2res}, we need also the following lemma.

\begin{lem}\label{Jacob}
Let $A,B\in k((u))((v))$, then for all pair of series $(x,y)$ associated with $(u,v)$ by a change of variables (\ref{cv}), we have
$$
(x,y)\res^2_{C,P}\left(dA\w dB \right)=0.
$$ 
\end{lem}

\begin{proof}
See appendix \ref{apJacob}.  
\end{proof}

\begin{proof}[Proof of theorem \ref{inv2res} if $\textrm{Char}(k)=0$.]
From lemma \ref{ypasdesouci}, we allready know that $1$-residues are invariant under (CV1). Thus, we will only study their behavior under (CV2). Consider any formal $2$-form
$$
\omega=\sum_{j=-l}^{-2} h_j(u)y^j du\w dy + \sum_{j\geq -1}h_j(u)y^j du\w dy=\omega_{-}+\omega_{\textrm{inv}}\ .
$$
From lemma \ref{polesimple}, the formal form $\omega_{\textrm{inv}}$
has an invariant $1$-residue under (\ref{cv}), thus so is its $2$-residue. We now have to study $\omega_{-}$.
Since extraction of $(x,y)$-$1$- and -$2$-residues are $k$-linear operations, we may only consider $2$-forms of the form
$$
\omega=\phi(u)du\w \frac{dy}{y^n}\quad \textrm{with}\quad \phi \in k((u))\quad \textrm{and}\quad n\geq 2.
$$

The formal $2$-form $\omega$ has a zero $(u,y)$-$2$-residue because its $(u,y)$-$1$-residue is also zero.
Then, we have to prove that its $(x,y)$-$2$-residue is zero too.
Before applying (CV2), we will work a little bit more on $\omega$.
First, isolate the term in $u^{-1}$ of the Laurent series $\phi$.
$$\phi(u)=\widetilde{\phi}(u)+\frac{\phi_{-1}}{u},\  \textrm{where}\ \widetilde{\phi}_i=
\left\{
\begin{array}{lll}
\phi_i & \textrm{if} & i\neq -1 \\
0 & \textrm{if} & i= -1.
\end{array}
\right.$$

\noindent The series $\widetilde{\phi}$ has a formal primitive $\widetilde{\Phi}$. Set $s:=\frac{1}{(1-n)y^{n-1}}$, which is a primitive of $1/y^n$ (it makes sense because $\textrm{Char}(k)$ is assumed to be zero). Then, we have
$$
\omega= d\widetilde{\Phi} \w ds+ \phi_{-1} \frac{du}{u}\w ds =\omega_{r}+\phi_{-1}\omega_{-1}.
$$
From lemma \ref{Jacob}, the form $\omega_{r}$ has a zero $2$-residue for all pair $(x,y)\in k((u))((v))^2$ connected to $(u,v)$ by a change of variables (\ref{cv}).
Now consider $\omega_{-1}=\frac{du}{u}\w \frac{dy}{y^n}$ and apply (CV2),
$$
\omega_{-1}=\frac{df(x,y)}{f(x,y)}\w \frac{dy}{y^n}.
$$

\noindent Recall that $f$ is of the form $\sum_{j\geq 0} f_j(x)y^j$ with
$$f_0(x)=f_{1,0}x+f_{2,0}x^2+\cdots \quad \textrm{and} \quad f_{1,0}\neq 0.$$ 
Thus, one can factorize $f_0$ in
$$
f_0(x)=f_{1,0}x\left(1+\frac{\displaystyle f_{2,0}}{\displaystyle f_{1,0}}x+\cdots \right).
$$
 
\noindent Set
$$
\begin{array}{lrcll}
 & r(x) & := & \frac{\displaystyle f_{2,0}}{\displaystyle f_{1,0}}x+\frac{\displaystyle f_{3,0}}{\displaystyle f_{1,0}}x^2 +\cdots & \in k[[x]]\\
\textrm{and} & \mu(x,y) & := & \frac{\displaystyle f_1(x)}{\displaystyle f_0(x)}y+\frac{\displaystyle f_2(x)}{\displaystyle f_0(x)}y^2 +\cdots & \in k((x))[[y]].
\end{array}
$$

\noindent The series $f$ has the following factorization
\begin{equation}\label{factf}
f(x,y)=f_{1,0}x(1+r(x))(1+\mu(x,y)). 
\end{equation}

\noindent Moreover, for every series $S$ in $xk[[x]]$ (resp. in $yk((x))[[y]]$) we define the formal logarithm of $1+S$ to be
$$\log(1+S):=\sum_{k=0}^{+\infty} (-1)^{k+1}\frac{S^k}{k}.$$

\noindent This makes sense because $\textrm{Char}(k)=0$
and this series converges for the $(x)$-adic (resp. $(y)$-adic) valuation.
Furthermore, $\frac{d(1+S)}{(1+S)}=d\log(1+S)$.
Using factorization (\ref{factf}), we obtain
$$
\omega_{-1}  =  \frac{ dx}{ x} \w \frac{ dy}{ y^n} +d\log(1+r)\w ds + d\log(1+\mu)\w ds.   
$$
From lemma \ref{Jacob}, second and third term of the sum have zero $(x,y)$-$2$-residues, and the first one has zero $(x,y)$-$1$-residue, hence a zero $(x,y)$-$2$-residue.
\end{proof}

\begin{proof}[Proof of theorem \ref{inv2res} in positive characteristic.]
  The idea is basically the same as in the proof of invariance of residues of $1$-forms on curves (c.f. \cite{sti} IV.2.9 or \cite{ser} prop II.7.5).
One proves that the $(x,y)$-$2$-residue of $\omega$ is a polynomial expression in a finite family of coefficients of $f$.
This polynomial has integer coefficients and depends neither on $f$ nor on the base field $k$. Thus, using the result of the proof in characteristic zero and the principle of prolongation of algebraic identities (\cite{bou} prop IV.3.9), we conclude that this polynomial is zero. For more details see appendix \ref{apinv2res}.   
\end{proof}



Consequently, from now on, when we deal with $2$-residues at $P$ along $C$, we won't have to precise the $(P,C)$-pair.

\section{Properties of residues}\label{propres}

\subsection*{Context}
In this section, $k$ is an \textbf{algebraically closed} field and $S$ a smooth geometrically integral quasi-projective surface over $k$. Moreover, $C$ denotes an irreducible absolutely reduced curve embedded in $S$ and $P$ a point of $C$.

\subsection{About $2$-residues}
Next lemma gives a necessary condition on $\omega$ to have nonzero $2$-residues at $P$ along $C$.

\begin{lem}\label{oukisont}
Let $\omega \in \Omega^2_{k(S)/k}$ having the curve $C$ as a pole. Let $P\in C$ such that $C$ is the only one pole of $\omega$ in a neighborhood of $P$. Then, $\res^2_{C,P}(\omega)=0$.
\end{lem}

\begin{proof}
Let $(u,v)$ be a strong $(P,C)$-pair and $n:=-val_C(\omega)$. There exists a function $h\in \OC$ such that
$$\omega=h du\w \frac{dv}{v^n}.$$
Furthermore, since $\omega$ has no pole but $C$ in a neighborhood of $P$, the function $h$ is in $\OP$.
Consequently, $h$ has a Taylor expansion $\sum_{i\geq 0} h_i(u)v^i $, where $h_i \in k[[u]]$ for all $i$. Then, $(u,v)res^1_{C}(\omega)=h_{n-1}(\bar{u})d\bar{u}$, which is regular at $P$, hence has zero residue at this point.
\end{proof}

\subsection{About $1$-residues}\label{about2}
We will give a new definition for one-codimensional residues generalizing the previous one. The goal is, as said in remark \ref{remglob}, to define $1$-residues as global objects on the curve $C$.

\begin{prop}\label{prores} 
Let $u,v$ be elements of $\OC$ such that $\bar{u}$ is a separating element\footnote{See \cite{sti} p. 127 for a definition.} of $k(C)/k$ and $v$ is a uniformizing parameter of $\OC$. Then, any $2$-form $\omega \in \Omega^2_{k(S)/k}$ can be expanded as 
\begin{equation}\label{Kudec}
\omega=\sum_{j\geq -l} f_j v^j du\w dv,
\end{equation}

\noindent where $f_j$'s are elements of the Hensel lift $\mathcal{K}_u$ of $k(C)$ over $k(u)$ in $\hOC$ (see proposition \ref{Ku}). 
Furthermore, the $1$-form $\bar{f}_{-1}d\bar{u}$ is rational on $C$ and does not depend on the choice of the uniformizing parameter $v$ of $\OC$.
\end{prop}

\begin{proof}
Recall that, from \cite{sch1} thm III.5.4.3, the space $\Omega^2_{k(S)/k}$ has dimension one over $k(S)$. Thus, there exists a unique function $f\in k(S)$ such that $\omega=fdu\w dv$. 
From corollary \ref{Ku[[v]]}, one can expand $f$ in $\mathcal{K}_u((v))$, which gives expansion (\ref{Kudec}).
From the construction of $\mathcal{K}_u$ (see proposition \ref{Ku}), $\bar{f}_{-1}$ may be identified to a rational function on $C$. Thus, the $1$-form $\bar{f}_{-1}d\bar{u}$ is rational on $C$.
To prove its independence on the choice of $v$, the reasoning is exactly the same as in the proof of lemma \ref{ypasdesouci}. 
\end{proof}

\begin{defn}\label{definres}
Under the assumptions of proposition \ref{prores}, we call $(u)$-$1$-residue of $\omega$ along $C$ and denote by $(u)\res^1_C(\omega)$ the rational $1$-form
$$
(u)\res^1_{C}(\omega):=\bar{f}_{-1}d\bar{u}\in \Omega^1_{k(C)/k}.
$$
\end{defn}

\begin{rem}
Using lemma \ref{polesimple}, one can prove that if $val_C(\omega)\geq -1$, then this $1$-form is also independent on the choice of $u$.  
\end{rem}

\begin{rem}
Let $\omega \in \Omega^2_{k(S)/k}$ and $u\in \hOC$ such that $\bar{u}$ is a separating element of $k(C)/k$. Set
$$
\mu:=(u)\res^1_C(\omega) \in \Omega^1_{k(C)/k}.
$$
Then, at each point $P$ of $C$ where $\bar{u}$ is a local parameter, we have
\begin{equation}\label{1,2}\tag{$\spadesuit$}
(u,v)\res^1_{C,P}(\omega)=\mu \quad \textrm{and} \quad \res^2_{C,P}(\omega)=
\res_P(\mu).
\end{equation}
\end{rem}

This remark asserts that definition \ref{definres} generalizes the notion one-codimensional residue (definition \ref{bordel}). Next proposition extends (\ref{1,2}) to any smooth point of $C$.

\begin{prop}\label{semiglob}
Let $\omega \in \Omega^2_{k(S)/k}$ and $u\in \OC$ such that $\bar{u}$ is a separating element
of $k(C)/k$, then at each smooth point $Q$ of $C$, we have 
$$
\res^2_{C,Q}(\omega)=\res_Q\left( (u)\res^1_C(\omega) \right).
$$
\end{prop}

\begin{rem}\label{avsing0}
In section \ref{secsing}, we generalize the definition of $2$-residue at $P$ along $C$ when $C$ may be singular at $P$.
Using this definition, the assumption ``$C$ is smooth at $Q$'' in proposition \ref{semiglob} can be cancelled (see remark \ref{avsing}).
\end{rem}

\begin{proof}[Proof of proposition \ref{semiglob}]
Set $\mu:=(u)\res^1_C(\omega)=\bar{f}_{-1}d\bar{u}$.

\medbreak

\noindent \textbf{Step 1.} Let $Q\in C$ at which $\bar{u}$ is regular and $(u-\bar{u}(Q),v)$ is a weak $(Q,C)$-pair. Set $u_0:=u-\bar{u}(Q)$. The function $\bar{u}$ is a local parameter of $\mathcal{O}_{C,Q}$.
Moreover, $\mathcal{K}_u=\mathcal{K}_{u_0}$ and $du=du_0$. Consequently, $(u_0,v)\res^1_{C,Q}(\omega)=\bar{f}_{-1}d\bar{u}_0=\mu$ and $\res^2_{C,Q}(\omega)=\res_Q(\mu)$.

\medbreak

\noindent \textbf{Step 2.} Let $Q\in C$ at which $\bar{u}$ is regular but $\bar{u}-\bar{u}(Q)$ is not a local parameter of $\mathcal{O}_{C,Q}$. Set $u_0:=u-\bar{u}(Q)$. We have $\omega=\sum_j f_j v^j du_0\w dv$, but $(u_0,v)$ is not a weak $(Q,C)$-pair. Let $(x,v)$ be a weak $(Q,C)$-pair. The function $\bar{x}$ is a local parameter of $\mathcal{O}_{C,Q}$ and for some $\phi \in k[[T]]$, we have
$$
\bar{u}_0=\phi(\bar{x})\ \textrm{in}\ k(C).
$$

\noindent Let $\sigma$ be the Hensel-lift of $\bar{x}$ in $\mathcal{K}_u$, last relation lifts in $\mathcal{K}_u$ and gives $u_0=\phi(\sigma)$. 
Consequently, we get a new formal expression for $\omega$,
\begin{equation}\label{sigdec}\tag{$\clubsuit$}
\omega=\sum_{j\geq -l}f_jv^j \phi'(\sigma)d\sigma \w dv.
\end{equation}
Notice that $\sigma \in \hOC$ and is congruent to $x$ modulo $(v)$. Therefore, $\sigma$ expands in $k((x))[[v]]$ as
$$
\sigma=x+\sigma_1(x)v+\sigma_2(x)v^2+\cdots
$$
Thus, the pair $(\sigma, v)$ is associated with $(x,v)$ by a change of variables (\ref{cv}). Using (\ref{sigdec}) and theorem \ref{inv2res}, we conclude that
$$
\res^2_{C,Q}(\omega)=
\res_P\left(\bar{f}_{-1}\phi'(\sigma)d\sigma \right)=
\res_P\left(\bar{f}_{-1}\phi'(\bar{x})d\bar{x}\right)=
\res_P(\mu)
.
$$

\medbreak

\noindent \textbf{Step 3.} Let $Q\in C$ at which $\bar{u}$ is not regular. Set $t:=1/u$ and notice that
$$
u=1/t\ \Rightarrow\ k(u)=k(t)\ \Rightarrow\ \mathcal{K}_u=\mathcal{K}_t.
$$
Thus, expansion of $\omega$ is of the form
$$\omega=\sum_{j\geq -l} f_jv^j \left(-\frac{dt}{t^2}\right)\w dv,$$
for some $v$ and
$$(t)\res^1_C (\omega)=-\bar{f}_{-1}\frac{d\bar{t}}{\bar{t}^2}=\mu.$$
Applying the arguments of the previous steps, we conclude the proof.
\end{proof}


\noindent \textbf{Summary.}
\begin{enumerate}
\item A $2$-residue depends only on a curve and a point.
Consequently, from now on, we will deal with $\res^2_{C,P}$ and not $(u,v)\res^2_{C,P}$ (definition \ref{bordel}).
\item A $1$-residue depends only on the curve and the choice of some element $u$ of $\hOC$, whose restriction to $C$ is a separating element of $k(C)/k$. Moreover this object gives a global information on $C$ and in a neighborhood of a point.
From now on, 
we will deal with $(u)\res_C^1$ (definition \ref{definres}) and not with $(u,v)\res^1_{C,P}$ (definition \ref{bordel}). We will also keep using map $\res^1_C$ for $2$-forms having $\mathfrak{m}_{S,C}$-adic valuation greater than or equal to $-1$.
\end{enumerate}

\begin{cor}\label{blowup}
Let $u$ be a function in $\OC$ whose restriction $\bar{u}$ to $C$ is a separating element of $k(C)/k$.
Let $\pi: \widetilde{S}\rightarrow S$ be the blowup
of $S$ at $P$ and $\widetilde{C}$ be the strict transform of $C$ by $\pi$. Then,
$$
(\pi^*u)\res^1_{\widetilde{C}}(\pi^* \omega) =\pi_{|\widetilde{C}}^* \left((u)\res^1_C(\omega) \right).
$$
\end{cor}

\begin{proof}
Surfaces $\widetilde{S}\smallsetminus E$ and $S\smallsetminus \{P\}$ are isomorphic under $\pi$. Furthermore, recall that $P$ is assumed to be a smooth point of $C$, thus $\pi$ induces an isomorphism between $\widetilde{C}$ and $C$.
The $1$-forms $(\pi^* u)\res^1_{\widetilde{C}}(\pi^* \omega)$ and $(u)\res^1_C(\omega)$ are pullback of each other by $\pi_{|\widetilde{C}}$ and its inverse.
\end{proof}

\begin{cor}\label{blowup2}
Let $(u,v)$ be a weak $(P,C)$-pair and $\pi: \widetilde{S}\rightarrow S$ be the blowup
of $S$ at $P$.
Denote by $\widetilde{C}$ the strict transform of $C$ by $\pi$ and by $Q$ the intersection point between $\widetilde{C}$ and the exceptional divisor. Then,
$$
\res^2_{\widetilde{C},Q}(\pi^*\omega)=\res^2_{C,P}(\omega).
$$
\end{cor}

\section{Residues along a singular curve}\label{secsing}

\subsection*{Context} The context of this section is that of sections \ref{onentwo}, \ref{lolo} and \ref{general} with only one difference, the curve $C$ may be singular at $P$.

\begin{prop}\label{prosing}
Let $\pi: \widetilde{S} \rightarrow S$ be a morphism obtained by a finite sequence of blowups of $S$ and such that the strict transform $\widetilde{C}$ of $C$ by $\pi$ is a desingularization of $C$ at $P$. Then, the sum
$$
\sum_{Q \in \pi^{-1} (\{P\})} \res^2_{\widetilde{C},Q}(\pi^* \omega)
$$
does not depend on the choice of the desingularization $\pi: \widetilde{S} \rightarrow S$.
\end{prop}

\begin{proof}
Let $\pi_1: \widetilde{S}_1 \rightarrow S$ and $\pi_2: \widetilde{S}_2 \rightarrow S$ be two morphisms as in the wording of the proposition. Denote by $\widetilde{C}_1$ and $\widetilde{C}_2$ the respective strict transforms of $C$ by these two morphisms.
Since both maps $\pi_1$ and $\pi_2$ induce desingularizations of $C$ at $P$, the point $P$ has the same number of preimages by $\pi_1$ and by $\pi_2$.
These preimages respectively denoted by $P_{1,1}, \ldots, P_{1,n}$ and $P_{2,1},\ldots, P_{2,n}$.
By construction of $\pi_1$ and $\pi_2$, there exists an open set $U_1 \subseteq \widetilde{C}_1$ (resp. $U_2 \subseteq \widetilde{C}_2$) containing $P_{1,1}, \ldots, P_{1,n}$ (resp. $P_{2,1}, \ldots, P_{2,n}$) and an isomorphism $\varphi: U_1 \rightarrow U_2$ such that ${\pi_1}_{|U_1}={\pi_2}_{|U_2}\circ \varphi$. Moreover,
for a suitable ordering of indexes, $\varphi$ sends $P_{1,i}$ on $P_{2,i}$ for all $i$.

Let $u$ be an element of $\OC$ whose restriction to $C$ is a separating element of $k(C)/k$. From corollary \ref{blowup}, the $1$-forms $(\pi_1^*u)\res^1_{\widetilde{C}_1}(\pi_1^* \omega)$ and $(\pi_2^*u)\res^1_{\widetilde{C}_2}(\pi_2^* \omega)$ are pullback of each other by $\varphi$ and $\varphi^{-1}$. Consequently,
$$
\forall i \in \{1, \ldots, n\},\ \res^2_{\widetilde{C}_1,P_{1,i}}(\pi_1^* \omega)=\res^2_{\widetilde{C}_2,P_{2,i}}(\pi_2^* \omega).
$$

 \noindent We conclude by adding last equalities for all $i$.
\end{proof}

\begin{defn}\label{singular}
  Under the assumptions of proposition \ref{prosing} the $2$-residue of a $2$-form $\omega \in \Omega^2_{k(S)/k}$ at $P$ along $C$ is defined by
$$
\res^2_{C,P}(\omega)=\sum_{Q \in \pi^{-1} (\{P\})} \res^2_{\widetilde{C},Q}(\pi^* \omega).
$$
\end{defn}

\begin{rem}\label{avsing}
As said in remark \ref{avsing0}, using definition \ref{singular} the statement of proposition \ref{semiglob} holds for singular points of $C$.
To prove this, apply the same arguments as in the proof of proposition \ref{semiglob} on a surface $\widetilde{S}$ such that there exists a map $\pi: \widetilde{S} \rightarrow S$ inducing a normalization of $C$.
\end{rem}

\section{Residue formulas}\label{RF}
We look for an analogous definition of the residue formula on curves (\cite{ser}  lem II.12.3 or \cite{sti} IV 3.3) in the two-dimensional case. We will give three statements about summations of $2$-residues.

\subsection*{Context} In this section, $k$ is an \textbf{algebraically closed} field and $S$ a smooth geometrically integral \textbf{projective} surface over $k$.

\begin{thm}[First Residue formula]\label{RF1}
Let $C$ be a reduced irreducible projective curve embedded in $S$. Then,
$$
\forall \omega \in \Omega^2_{k(S)/k},\ \sum_{P \in C} \res^2_{C,P}(\omega)=0.
$$  
\end{thm}

\begin{proof}
Let $u$ be an element of $\OC$ whose restriction $\bar{u}$ to $C$ is a separating element of $k(C)$.
If $C$ is smooth, then apply proposition \ref{semiglob} and the classical residue formula on curves to $(u)\res^1_C(\omega)$. Else, use definition \ref{singular} and apply the same arguments to a morphism $\phi:\widetilde{S} \rightarrow S$ inducing a normalization of $C$. 
\end{proof}

\begin{rem}
If $val_C(\omega)\geq -1$, then from proposition \ref{inv1res} and definition \ref{dinv1res}, the $2$-form $\omega$ has a $1$-residue along $C$ denoted by $\res^1_C(\omega)$. Thus, in this particular situation, last theorem is an easy consequence of the classical residue formula on curves applied to the $1$-form $\res^1_C(\omega)$. The nonobvious part of this proposition is that the statement holds even if $val_C(\omega)<-1$.  
\end{rem}

\begin{thm}[Second residue Formula]\label{RF2}
Let $\mathcal{C}_{S,P}$ be the set of germs of irreducible reduced curves embedded in $S$ and containing $P$. Then,
$$
\forall \omega \in \Omega^2_{k(S)/k},\ \sum_{C\in \mathcal{C}_{S,P}} \res^2_{C,P}(\omega)=0.
$$  
\end{thm}

\begin{rem}
  Notice that this sum is actually finite because almost all $C\in \mathcal{C}_{S,P}$ is not a pole of $\omega$ thus the $2$-residue at $P$ along this curve is zero.
\end{rem}

\begin{proof}
Let $\omega\in \Omega^2_{k(S)/k}$ and $C_1,\ldots,C_n \in \mathcal{C}_{S,P}$ be the set of its poles in a neighborhood of $P$. 
We will prove the theorem by induction on $n$.

\smallbreak

\noindent \textbf{Step 1.} Assume that, for each pair of curves $C_i,C_j$ with $i\neq j$, their intersection multiplicity at $P$ is one.

\smallbreak

\noindent \textbf{If $\mathbf{n=1}$.} From lemma \ref{oukisont} the sum is obviously zero.

\smallbreak

\noindent \textbf{If $\mathbf{n=2}$.} Let $u_1,u_2$ be respectively local equations of $C_1$ and $C_2$. Then, $(u_1,u_2)$ is a strong $(P,C_2)$-pair and $(u_2,u_1)$ a strong $(P,C_1)$-pair, because $C_1$ and $C_2$ are assumed to have a normal crossing at $P$. Thus, for some $h\in \OP$ and some positive integers $n_1$ and $n_2$, we have
$$\omega=h\ \frac{du_1}{u_1^{n_1}}\w \frac{du_2}{u_2^{n_2}}.$$

\noindent Expand $h$ in Taylor series $h=\sum h_{ij}u_1^iu_2^j$. Using the anticommutativity of the external product,
a brief computation gives $\res^2_{C_2,P}(\omega)=-\res^2_{C_1,P}(\omega)=h_{n_1-1,n_2-1}$.

\smallbreak

\noindent \textbf{If $\mathbf{n\geq 2}$.} Consider $\pi:\ \widetilde{S}\rightarrow S$ the blowup of $S$ at $P$. Denote by $E$ the exceptional divisor, by $\widetilde{C}_i$ the strict transform of $C_i$ and by $Q_i$ the intersection point between $E$ and $\widetilde{C}_i$. Points $Q_i$'s are all distinct and curves $E$ and $\widetilde{C}_i$ have normal crossing at $Q_i$. 
 The curve $E$ is projective and the $\widetilde{C}_i$'s are the only poles of $\pi^* \omega$ which cross $E$. Furthermore, the previous case entails $\res^2_{\widetilde{C}_i,Q_i}(\pi^* \omega)=-\res^2_{E,Q_i}(\pi^* \omega)$ for all $i$. Consequently,
from corollary \ref{blowup2}, we have
$$
\sum_{i=1}^n \res^2_{C_i,P}(\omega)=\sum_{i=1}^n \res^2_{\widetilde{C_i},Q_i}(\pi^* \omega)=-\sum_{i=1}^n \res^2_{E,Q_i}(\pi^*\omega)
$$

\noindent and last sum is zero from theorem \ref{RF1}.

\smallbreak

\noindent \textbf{Step 2.} In the general case, a curve $C_i$ might be singular at $P$ or intersect the other $C_j$'s with higher multiplicity. After a finite number of blowups, using definition \ref{singular} and applying same arguments to the resolution tree,  we get the expected result.
\end{proof}

\begin{rem}
Notice that the valuation of $\pi^* \omega$ along the exceptional divisor $E$ is not always greater than or equal to $-1$. This valuation is given by the formula
$$
val_E(\pi^* \omega)=1+\sum_{C\in \mathcal{C}_{S,P}} val_C(\omega),
$$
where the set $\mathcal{C}_{S,P}$ is that of theorem \ref{RF2}.
For a proof of this formula see \cite{H} prop V.3.3 and V.3.6.
Therefore, theorem \ref{RF1} is necessary to conclude in the first step of last proof.
\end{rem}

\noindent To state the third residue formula, we need to extend the definition of $2$-residues at a point along a curve to $2$-residues at a point along a divisor.

\begin{defn}
  Let $D=n_1 C_1+\cdots+n_p C_p$ be a divisor on $S$ and $\omega \in \Omega^2_{k(S)/k}$. We define the $2$-residue of $\omega$ at $P$ to be
$$
\res^2_{D,P}(\omega):=\sum_{i=1}^n\res^2_{C_i,P}(\omega).
$$
\end{defn}

\begin{rem}
Notice that coefficients $n_i$'s of $D$ are not involved in this definition.
Actually, $\res^2_{D,P}$ depends only on the support of the $D$. The most logic notation would have been ``$\ \! \res^2_{\supp (D)}$'' which is too heavy.
\end{rem}

\begin{thm}[Third residue formula]\label{RF3}
Let $D_a, D_b$ be two divisors such that the set $\supp(D_a)\cap \supp (D_b)$ is finite.
Let $\Delta$ be the zero-cycle given by the scheme-theoretic intersection $\Delta:=D_a \cap D_b$. Set $D:=D_a+D_b$.
Then,
$$
\forall \omega \in \Omega^2(-D),\ \sum_{P\in \supp \Delta} \res^2_{D_a, P}(\omega)=
\sum_{P\in S}\res^2_{D_a,P}(\omega) =0.
$$
\end{thm}

\begin{proof}
From theorem \ref{RF1}, for each irreducible component $C_i$ of the support of $D_a$, we have $\sum_{P\in C_i}\res^2_{C_i,P}(\omega)=0$. From theorem \ref{RF2}, if a point $P$ is out of the support of $\Delta$, then $\res^2_{D_a,P}(\omega)=0$. A combination of both claims concludes the proof.
\end{proof}

\begin{rem}\label{symsym}
From theorem \ref{RF2} and under the assumptions of theorem \ref{RF3}, for all $P$ in $S$, we have
$$
\res^2_{D_a,P}(\omega)=-\res^2_{D_b,P}(\omega).
$$ 
Consequently, the statement in theorem \ref{RF3} holds replacing $D_a$ by $D_b$.
\end{rem}

\part{Application to coding theory}\label{partiecodes}



\section*{Notations}
Let $X$ be a variety defined over finite a field $k$, we denote by $\textrm{Div}_{k}(X)$ the group of rational Weil divisors on $X$. That is, the free abelian group spanned by irreducible one-codimensional closed subvarieties of $X$.
If $G \in \textrm{Div}_{k}(X)$, then we use the following notations.
\begin{enumerate}
\item $G^+$ denotes the effective part of $G$.
\item $L(G)$ denotes the Riemann-Roch space of rational functions
$$
L(G):= \left\{ f \in k(X),\ (f)+G\geq 0  \right\} \cup \{ 0 \}.
$$
\item $\Omega^2(G)$ denotes the Riemann-Roch space of rational $2$-forms
$$
\Omega^2(G):= \left\{ \omega \in \Omega^2_{k(X)/k},\ (\omega)-G\geq 0  \right\} \cup \{ 0 \}.
$$
\item If $G' \in \textrm{Div}_{k}(X)$ such that the supports of $G$ and $G'$ have no common component in a neighborhood of $P$, we denote by $m_P(G,G')\in \Z$ the intersection multiplicity of these divisors at $P$.
\end{enumerate}

\section{About codes from curves, classical constructions}\label{codesoncurves}
In this section $C$ is a smooth projective absolutely irreducible curve over a finite field $\F_q$. Let $G$ be a rational divisor on $C$ and $P_1,\ldots, P_n$ be a family of rational points of $C$ avoiding the support of $G$. Set $D:=P_1+\cdots +P_n \in \textrm{Div}_{\F_q}(C)$ and
$$
\textrm{ev}_D: \left\{
\begin{array}{ccl}
L(G) & \rightarrow & \F_q^n \\    
 f & \mapsto & {(f(P_i))}_{i=1\ldots n}
\end{array}
\right.
,\ \ 
\res_D: \left\{
\begin{array}{ccl}
\Omega^1(G-D) & \rightarrow & \F_q^n \\    
\omega & \mapsto & {(\res_{P_i}(\omega))}_{i=1\ldots n}
\end{array}.
\right. 
$$  
We define the codes $C_L(D,G):=\textrm{Im}(\textrm{ev}_D)$ and 
$C_{\Omega}(D,G):=\textrm{Im}(\res_D)$ called respectively functional code and differential code.

\noindent Both constructions are linked by the following properties.
\begin{enumerate}
\item[\textbf{(OR):}]\label{orth} $C_{\Omega}(D,G)=C_L(D,G)^{\bot}$.
\item[\textbf{(L$\mathbf{\Omega}$):}]\label{diff=func} For some canonical divisor $K$, we have $C_{\Omega}(D,G)=C_L(D,K-G+D)$.
\end{enumerate}
See \cite{step},  \cite{sti} or \cite{TV} for the proofs of these statements.
Relation \textbf{(OR)} is a consequence of the residue formula for inclusion ``$\subseteq$'' and of Riemann-Roch's theorem for the reverse one. This relation is used in almost all algebraic decoding algorithms (see \cite{hohpel}).
Relation \textbf{(L$\mathbf{\Omega}$)} is a consequence of the weak approximation theorem (\cite{sti} thm I.3.1). It allows to restrict the study of algebraic-geometric codes to only one class, for example functional codes which seems to be easier to study.
The goal of this second part is to extend some of these statements to surfaces.

\section{Algebraic-geometric codes on surfaces}\label{codesonsurfaces}

From now on, $S$ denotes a smooth geometrically integral projective surface over $\F_q$ and $\overline{S}:=S\times_{\F_q}\overline{\F}_q$. Moreover, $G$ denotes a rational divisor on $S$ and $P_1,\ldots , P_n$ a set of rational points of $S$ avoiding the support of $G$.
Set $\Delta:=P_1+\cdots +P_n$. Notice that $\Delta$ is not a divisor but a $0$-cycle.
Most of the difficulties we will meet come from this difference of dimension between $G$ and $\Delta$.

\subsection{Functional codes}
As said in the introduction, the functional construction of codes extends to higher-dimensional varieties (see \cite{man} I.3.1). Define the map
$$
\textrm{ev}_{\Delta}:
\left\{
\begin{array}{ccl}
L(G) & \rightarrow & \F_q^n \\
f & \rightarrow & (f(P_1),\ldots ,f(P_n)).   
\end{array}
\right.
$$
The functional code is $C_L(\Delta,G):=\textrm{Im}(ev_{\Delta})$.
The study of such codes is really more complicated than that of codes on curves.
Particularly, finding a minoration of the minimal distance becomes a very difficult problem.
For more details about this topic, see references cited in introduction.
\subsection{Differential codes}

To define differential codes, we need more than $G$ and $\Delta$. We want to evaluate $2$-residues of some rational differential forms with prescribed poles.
Unfortunately, $2$-residues depend not only on a point but on a flag $P\in C\subset S$.
 Thus, we have to input another divisor.

\begin{defn}
  Let $D \in \textrm{Div}_{\F_q}(S)$ and assume that $D$ is the sum of two divisors $D_a,D_b$ whose supports have no common irreducible component. Then, one can define the map
$$
\res^2_{D_a,\Delta}:
\left\{
\begin{array}{ccl}
\Omega^2(G-D) & \rightarrow & \F_q^n \\
\omega & \mapsto & (\res^2_{D_a,P_1}(\omega),\ldots ,
\res^2_{D_a,P_n}(\omega)).
\end{array}
\right.
$$
The differential code is defined by $C_{\Omega}(\Delta,D_a,D_b,G):=\textrm{Im}(\res^2_{D_a,\Delta})$.
\end{defn}

\begin{rem}
We can also define a map $\res^2_{D_b,\Delta}$, but from theorem \ref{RF2}, we have 
$\res^2_{D_b,\Delta}=-\res^2_{D_a,\Delta}$. Thus, both maps have the same image and 
$$
C_{\Omega}(\Delta,D_a,D_b,G)=C_{\Omega}(\Delta,D_b,D_a,G).
$$  
\end{rem}

\subsection{$\Delta$-convenience}
Actually, if one chooses an arbitrary divisor $D$, last definition is not very convenient.
Recall that, from lemma \ref{oukisont} and theorem \ref{RF2}, $\res^2_{D_a,P_i}(\omega)$ is nonzero only if the supports of $D_a^+$ and $D_b^+$ intersect at $P_i$.
Therefore, if we want to have a code which is linked to the functional code $C_L(\Delta,G)$, the divisor $D$ must be itself \textit{related} to the $0$-cycle $\Delta$.
We will first define the notion of $\Delta$-convenient pair of divisors. Afterwards, we will give a criterion of $\Delta$-convenience. Although this one may look ugly, it is actually easy to handle.

\medbreak

Let $D_a,D_b$ be a pair of $\F_q$-rational divisors on $S$ whose supports have no common component and set $D:=D_a+D_b$. From now on, $\mathcal{F}$ denotes the sheaf on $\overline{S}$ defined by
$$
\mathcal{F}(U)=\left\{\omega \in \Omega^2_{\overline{\F}_q(\overline{S})/\overline{\F}_q},\ (\omega_{|U})\geq -D_{|U}  \right\}.
$$
Moreover, for all point $P\in \overline{S}$, the stalk of $\mathcal{F}$ at $P$ is denoted by $\mathcal{F}_P$ (see \cite{H} II.1 p. 62 for definition of ``stalk'').
Notice that $H^0(\overline{S},\mathcal{F})=\Omega^2(-D)\otimes_{\F_q}\overline{\F}_q$.

\begin{defn}\label{Deltac}
The pair $(D_a,D_b)$ is said to be $\Delta$-convenient if it satisfies the following conditions.
\begin{enumerate}
\item[$(i)$] Supports of $D_a$ and $D_b$ have no common irreducible components.
\item[$(ii)$] For all point $P\in \overline{S}$, the map $\res^2_{D_a,P}:\ \mathcal{F}_P \rightarrow \overline{\F}_q$ is $\mathcal{O}_{\overline{S},P}$-linear.
\item[$(iii)$] This map is surjective for all $P\in \supp (\Delta)$ and zero elsewhere. 
\end{enumerate}
\end{defn}

\begin{rem}
The structure of $\mathcal{O}_{\overline{S},P}$-module of $\overline{\F}_q$ is induced by the morphism $f \rightarrow f(P)$. Thus, if the map $\res^2_{D_a,P}$ satisfies $(ii)$, then it vanishes on $\mathfrak{m}_{\overline{S},P} \mathcal{F}_P$. 
\end{rem}

\begin{rem}\label{Da-Db}
Let $P\in \overline{S}$ and $\omega \in \mathcal{F}_P$. From remark \ref{symsym}, we have
$$\res^2_{D_a,P}(\omega)=-\res^2_{D_b,P}(\omega).$$ 
Consequently, if at a point $P\in \overline{S}$, the map $\res^2_{D_b,P}$ is $\OP$-linear and satisfies $(ii)$ and $(iii)$, so does $\res^2_{D_a,P}$.
\end{rem}

\begin{prop}[Criterion for $\Delta$-convenience]\label{crit}
Let $(D_a,D_b)$ be a pair of $\F_q$-rational divisors having no common irreducible component on $S$ and set $D:=D_a+D_b$. If $(D_a,D_b)$ satisfies the following conditions, then it is a $\Delta$-convenient pair. 
\begin{enumerate}
\item For each $P\in \supp (\Delta)$, there exists an irreducible curve $C$ smooth at $P$ such that in a neighborhood $U_P$ of $P$, either ${(D_a^+)}_{|U_P}$ or ${(D_b^+)}_{|U_P}$ equals $C_{|U_P}$ and $m_P(C,D-C)=1$.
\item\label{2} For each $P\in \overline{S}\smallsetminus \supp (\Delta)$, either $D_*=D_a$ or $D_*=D_b$ satisfies the following conditions.
For each $\overline{\F}_q$-irreducible component $\overline{C}$ of $D_*^+$ containing $P$,
  \begin{enumerate}
  \item\label{a} the curve $\overline{C}$ is smooth at $P$,
  \item\label{b} this curve $\overline{C}$ appears in $D_*$ with coefficient one,
  \item\label{c} $m_P(\overline{C}, D-\overline{C})\leq 0$.
  \end{enumerate}
\end{enumerate}
\end{prop}

\begin{rem}
In condition (\ref{2}) of this criterion, the divisor $D_*$ may be zero in a neighborhood of $P$ (actually that is what happens at almost all point $P$). In this situation, conditions (\ref{a}), (\ref{b}) and (\ref{c}) are obviously satisfied.
\end{rem}

\noindent For the proof of this proposition, we need next lemma and its corollary.

\begin{lem}\label{valres}
Let $\overline{C}$ be an irreducible curve over $\overline{\F}_q$ embedded in $\overline{S}$ and $P$ be a smooth point of $\overline{C}$. Let $\omega \in \Omega^2_{\overline{\F}_q(\overline{S})/\overline{\F}_q}$ having a simple pole along $\overline{C}$. Then
$$
val_P(\res^1_{\overline{C}}(\omega))=m_P(\overline{C},(\omega)+\overline{C}),
$$  
where $val_P$ denotes the valuation at the point $P$ on $\Omega^1_{\overline{\F}_q(\overline{C})/\overline{\F}_q}$.
\end{lem}

\begin{proof}
Let $\varphi$, $\psi$ and $v$ be respective local equations of $\left((\omega)+\overline{C}\right)^+,\ \left((\omega)+\overline{C}\right)^-$ and $\overline{C}$ in a neighborhood of $P$. Let $u\in \overline{\F}_q(\overline{S})$ such that $(u,v)$ is a strong $(P,\overline{C})$-pair, then for some $h\in \mathcal{O}_{\overline{S},P}^{\times}$, we have
$$
\omega=h\frac{\varphi}{\psi} du\w \frac{dv}{v}.
$$
Thus, $\res^1_{\overline{C}}(\omega)=\bar{h}\bar{\varphi} \bar{\psi}^{-1}d\bar{u}$ and since $\bar{h}\in \mathcal{O}_{\overline{C},P}^{\times}$, we have $val_P(\bar{h}d\bar{u})=0$. Consequently,
$$
val_P(\res^1_{\overline{C}}(\omega))\ = \ val_P(\bar{\varphi})-val_P(\bar{\psi}).
$$ 
Furthermore, 
$$
\begin{array}{rcl}
m_P(\overline{C},(\omega)+\overline{C}) & = & m_P(\overline{C},((\omega)+\overline{C})^+)-m_P(\overline{C},((\omega)+\overline{C})^-) \\
 & = & \dim_{\overline{\F}_q}\mathcal{O}_{\overline{S},P} / (\varphi,v) - \dim_{\overline{\F}_q}\mathcal{O}_{\overline{S},P} / (\psi,v) \\
 & = & \dim_{\overline{\F}_q} \mathcal{O}_{\overline{C},P}/ (\bar{\varphi}) - \dim_{\overline{\F}_q} \mathcal{O}_{\overline{C},P}/ (\bar{\psi}) \\
 & = & val_P(\bar{\varphi}) - val_P(\bar{\psi}).   
\end{array}
$$
\end{proof}

\begin{cor}\label{linear}
Let $\overline{C}$ be an irreducible curve embedded in $\overline{S}$ and $P$ be a smooth point of $\overline{C}$.
Let $\omega \in \Omega^2_{\overline{\F}_q(\overline{S})/\overline{\F}_q}$ such that $val_{\overline{C}}(\omega)
\geq -1$ and $m_P(\overline{C}, (\omega)+\overline{C})\geq -1$. Then,
$$
\forall f\in \mathcal{O}_{\overline{S},P},\ \res^2_{\overline{C},P}(f\omega)=f(P)\res^2_{\overline{C},P}(\omega).
$$ 
\end{cor}

\begin{proof}
Let $(u,v)$ be a strong $(P,\overline{C})$-pair and $f$ be an element of $\mathcal{O}_{\overline{S},P}$. Since $val_{\overline{C}}(\omega)\geq -1$, there exists $\psi \in \mathcal{O}_{\overline{S},\overline{C}}$ such that
$$
\omega=\psi du\w \frac{dv}{v}.
$$

\noindent Set $\mu:=\res^1_{\overline{C}}(\omega)=\bar{\psi}d\bar{u}$. The condition $val_{\overline{C}}(\omega)\geq -1$ entails also
$$
\res^1_{\overline{C}}(f\omega)=\bar{f}\bar{\psi}d\bar{u}=\bar{f}\mu.
$$

\noindent From lemma \ref{valres}, we have $val_P(\mu)=m_P(\overline{C},(\omega)+\overline{C})\geq -1$.
Thus,
$$
\res^2_{\overline{C},P}(f\omega)=\res_P(\bar{f}\mu)=\bar{f}(P)\res_P(\mu)
=f(P)\res^2_{\overline{C},P}(\omega).
$$  
\end{proof}

\begin{proof}[Proof of proposition \ref{crit}]
Let $(D_a,D_b)$ be a pair of divisors satisfying conditions of proposition \ref{crit}. Condition $(i)$ of definition \ref{Deltac} is obviously satisfied, because supports of $D_a$ and $D_b$ are assumed to have no common irreducible component.
Now, we prove that $(ii)$ and $(iii)$ are satisfied.
First, recall that $\mathcal{F}$ denotes the sheaf of rational $2$-forms $\omega$ on $\overline{S}$ satisfying locally $(\omega)\geq -D=-D_a-D_b$.

\medbreak

\noindent \textbf{Condition $(ii)$.} Let $P\in \supp (\Delta)$ and $\omega \in \mathcal{F}_P$, where $\mathcal{F}_P$ denotes the stalk of the sheaf $\mathcal{F}$ at $P$. From $(1)$, there is an irreducible curve $C$, smooth at $P$ such that either $D_a^+$ or $D_b^+$ equals $C$ in a neighborhood of $P$.
Using remark \ref{Da-Db}, we may assume that $D_a^+=C$ without loss of generality. Thus, $C$ is the only one irreducible component of $\supp (D_a^+)$ in a neighborhood of $P$.
Therefore, $\res^2_{D_a,P}(\omega)=\res^2_{C,P}(\omega)$.
Consequently, $val_C(\omega)\geq -1$ and corollary \ref{linear} asserts that $\res^2_{C,P}$ (hence $\res^2_{D_a,P}$) is $\OP$-linear.

\medbreak

\noindent \textbf{Condition $(iii)$.} Let $P\in \overline{S}$ be a point out of the support of $\Delta$. From remark \ref{Da-Db}, we may assume without loss of generality that condition $(2)$ in proposition \ref{crit} is satisfied by $D_a$ (i.e. $D_*=D_a$ at $P$).
Let $\overline{C}$ be an $\overline{\F}_q$-irreducible
component of $\supp(D_a^+)$ and $\omega \in \mathcal{F}_P$. From (\ref{b}), $val_{\overline{C}}(\omega)\geq -1$ and from
lemma \ref{valres} and (\ref{c}), we have $val_P(\res^1_{\overline{C}}(\omega)) \geq 0$. Consequently, $\res^2_{\overline{C},P}(\omega)=0$, which concludes the proof. 
\end{proof}


\begin{exmp}\label{p2}
Let $S=\P^2$ and let $\Delta$ be the sum of the rational points of an affine chart $U$. Let $x,y$ be affine coordinates on $U$. For all $\alpha, \beta \in \F_q$, set $D_{a,\alpha}$ the line $\{ x=\alpha \}$ and $D_{b,\beta}:=\{y=\beta\}$. Now, set $D_a:=\sum_{\alpha \in \F_q}D_{a,\alpha}$ and $D_b:=\sum_{\beta} D_{b,\beta}$. The pair $(D_a,D_b)$ satisfies the criterion of proposition \ref{crit}, hence is $\Delta$-convenient.
Notice that the components of $D_a$ (resp.$D_b$) intersect themselves at a point lying on the line \textit{at infinity}, which does not represent any contradiction with the definition of $\Delta$-convenience.   
\end{exmp}

Notice that in the definition, neither $D_a$ nor $D_p$ are assumed to be effective. In some situation it is necessary to use noneffective divisors. This happens in next example.

\begin{exmp}
Consider again $S=\P^2$ and assume that the base field is $\F_q$ with $q$ odd. Set $\Delta=P_1+P_2+P_3$ with $P_1=(0:0:1)$, $P_2=(1:0:1)$ and $P_3=(0:1:1)$.  
The pair $(D_a,D_b)$ defined by
$D_a=\{Y=0\}+\{Y=1\}$ and $D_b=\{X=0\}+\{X=1\}-\{X+Y-2=0\}$, is $\Delta$-convenient. 
However, in this situation, there does not exist any $\Delta$-convenient pair of effective divisors. The proof of last claim is left to the reader.
\end{exmp}



\section{Properties of differential codes}\label{propcodes}
\subsection{Orthogonality}\label{suborth}

\begin{thm}\label{inc1}
Let $(D_a,D_b)$ be a $\Delta$-convenient pair and set $D:=D_a+D_b$, then
$$
C_{\Omega}(\Delta,D_a,D_b,G)\subseteq C_L(\Delta,G)^{\bot}.
$$  
\end{thm}

\begin{proof}
Let $\omega \in \Omega^2(G-D)$ and $f\in L(G)$, then $f\omega \in \Omega^2(-D)$ and from the definition of $\Delta$-convenient pairs,
$$
\forall P\in \overline{S},\ 
\res^2_{D_a,P}(f\omega)=\left\{
\begin{array}{ccl}
0 & \textrm{if} & P\notin \supp(\Delta) \\
f(P)\res^2_{D_a,P}(\omega) & \textrm{if} & P\in \supp (\Delta).   
\end{array}
\right.
$$

\noindent Thus,
$$<\textrm{ev}_{\Delta}(f),\res^2_{D_a,\Delta}(\omega)>= \sum_{P\in \supp (\Delta)} f(P)\res^2_{D_a,P}(\omega)=\sum_{P\in \supp (\Delta)}\res^2_{D_a,P}(f\omega).$$
And last sum is zero from theorem \ref{RF3}.
\end{proof}

In section \ref{examples}, we prove that in some situation, the reverse inclusion is false for any choice of a $\Delta$-convenient pair of divisors. Thus, in general, we do not have equality.
An interpretation of this statement is that, even if a pair of $\Delta$-convenient divisors is linked to $\Delta$, it is not involved in the functional construction.
This lack of canonicity in the choice of $D$ might be the reason of this non-equality.
In a forthcoming paper, we will study how to get the whole orthogonal of a functional code, using differentials.

\subsection{A differential code is functional}
Recall that in section \ref{codesoncurves}, we discussed about two relations denoted by \textbf{(OR)} and \textbf{(L$\mathbf{\Omega}$)}.
We just said that it is not possible to extend perfectly the orthogonality relation \textbf{(OR)}. Nevertheless, next proposition asserts that relation \textbf{(L$\mathbf{\Omega}$)} holds on surfaces, a differential code is always a functional one associated with some canonical divisor. Recall that, the proof of \textbf{(L$\mathbf{\Omega}$)} for curves is a consequence of the weak approximation theorem. Here is the needed statement for surfaces.

\begin{prop}\label{approx}
Let $P_1,\ldots, P_m$ and $Q_1, \ldots, Q_n$ be two families of closed points of $S$ and $C$ be an irreducible curve embedded in $S$. Suppose that the $P_i$'s are contained in $C$ and the $Q_i$'s are out of it. Then, there exists a function $u\in \F_q(S)$ satisfying the following conditions.
\begin{enumerate}
\item[$(i)$] $\forall i \in \{1,\ldots ,m\}$, $u$ is a local equation of $C$ in a neighborhood of $P_i$.
\item[$(ii)$] $\forall j \in \{1,\ldots ,n\},\ Q_j \notin\supp (u)$, i.e. $u\in \mathcal{O}_{S,Q_i}^{\times}$
\end{enumerate}
\end{prop}
 
\begin{proof}
Choose $u_0$, a uniformizing parameter of $\OC$. Then, $(u_0)=C+D$ where $D\in \textrm{Div}_{\F_q}(S)$ whose support does not contain $C$. From the \textit{moving lemma} (\cite{sch1} thm III.1.3.1), there exists a divisor $D'$ linearly equivalent to $D$ whose support avoids $P_1,\ldots , P_m, Q_1, \ldots , Q_n$. Thus, for some function $f \in \F_q (S)$, we have $D'=D+(f)$ and $u:=fu_0$ is a solution of the problem.  
\end{proof}

\noindent \textbf{N.B.} In the whole book of Shafarevich \cite{sch1} the base field is assumed to be algebraically closed. Nevertheless the very same proof holds over an arbitrary field.

\begin{cor}\label{omega0}
Let $(D_a,D_b)$ be a $\Delta$-convenient pair and set $D:=D_a+D_b$, then there exists a differential $\omega_0\in \Omega^2_{k(S)/k}$ satisfying the following conditions.
\begin{enumerate}
\item\label{vvv} For some open set $U$ containing $\supp (\Delta)$, we have
$({\omega_0}_{|U})=-D_{|U}$.
\item\label{vvx} $\forall P\in \supp (\Delta),\ \res^2_{D_a,P}(\omega_0)=1$.
\item\label{vxx} $\forall P\in \supp (\Delta),\ \forall f\in \OP,\
\res^2_{D_a,P}(f\omega_0)=f(P)\res^2_{D_a,P}(\omega_0)$.
\end{enumerate}
\end{cor}

\begin{proof}
Let $X_1,\ldots, X_r$ and $Y_1,\ldots ,Y_s$
be respectively the irreducible components of
$\supp (D_a)$ and $\supp (D_b)$. 
That is $D_a=m_1 X_1+\cdots +m_r X_r$ and $D_b=n_1Y_1+\cdots +n_s Y_r$ for some integers $m_i$'s and $n_j$'s.  
From proposition \ref{approx} there is an open subset $U$ of $S$ containing the support of $\Delta$ and functions $u_1,\ldots, u_r,v_1,\ldots,v_s$ such that $u_i$ (resp $v_j$) is an equation of $X_i$ (resp. $Y_j$) in $U$.
Set $u:=\prod_i u_i^{m_i}$ and $v:= \prod_i v_i^{n_i}$.

\noindent Let $\mu$ be a rational $2$-form on $S$ having neither zeros nor poles in a neighborhood of the support of $\Delta$ and set
$$
\omega_0:= \frac{\mu}{uv}.
$$

\noindent Replacing $U$ by a smaller open set containing $\supp (\Delta)$, we may assume that
$\mu$ has neither zeros nor poles in $U$. Thus, statements (\ref{vvv}) and (\ref{vxx}) are satisfied by $\omega_0$.
Moreover, from the definition \ref{Deltac} of $\Delta$-convenience, we have
$$
\forall P\in \supp (\Delta),\ \res^2_{D_a,P}(\omega_0)=a_P\neq 0.
$$
Choose $g\in \cap_{P\in \supp (\Delta)}\OP^{\times}$ such that $g(P)=a_P^{-1}$ for all $P\in \supp (\Delta)$. Then, replacing $U$ by a smaller open set containing $\supp (\Delta)$ and $\omega_0$ by $g\omega_0$, the three conditions are satisfied.
\end{proof}

\begin{thm}\label{dif=fonc}
Let $D=D_a+D_b$ such that $(D_a,D_b)$ is $\Delta$-convenient. There exists a canonical divisor $K$ such that
$$
C_{\Omega}(\Delta,D_a,D_b, G)=C_L(\Delta, K-G+D).
$$  
\end{thm}

\begin{proof}
From corollary \ref{omega0}, there exists a $2$-form $\omega_0$ satisfying (\ref{vvv}), (\ref{vvx}) and (\ref{vxx}). Set $K:=(\omega_0)$, this divisor is of the form $K=-D+R$ where the support of $R$ avoids that of $\Delta$.
Let $\omega \in \Omega^2(G-D)$, then for some function $f\in L(K-G+D)$, we have $\omega=f\omega_0$. Notice that $K-G+D=G+R$, then any function $f\in L(K-G+D)$ is regular in a neighborhood of each $P\in \supp (\Delta)$.   
Consequently, from condition (\ref{vxx}) in corollary \ref{omega0}, we have
$$
\res^2_{D_a,\Delta}(\omega)=\res^2_{D_a,\Delta}(f\omega_0)=ev_{\Delta}(f).
$$
\end{proof}

\noindent Any differential code is actually a functional one. Notice that, if the converse statement is trivial for codes on curves, it is less easy in our situation. Indeed, to prove that a functional code is differential, we have to build a $\Delta$-convenient pair of divisors.

\subsection{Converse statement, a functional code is differential}

\begin{lem}\label{exconv}
Let $Q_1,\ldots, Q_m$ be rational points of $S$ and set $\Gamma:=Q_1+\cdots+Q_m$. Then, there exists a $\Gamma$-convenient pair $(D_a, D_b)$.  
\end{lem}

\begin{proof}
\textbf{Step 1: Construction of $\mathbf{D_a}$.} 
Choose a curve $C$ (which may be reducible) containing the whole support of $\Gamma$ and regular at each point of it and set $D_a:=\sum_k C_k$ where $C_k$'s are the irreducible components of $C$.
Finding such a curve is an interpolation problem with infinitely many solutions.

\medbreak

\noindent \textbf{Step 2: Construction of $\mathbf{D_b}$.} Choose another divisor $D'$ interpolating all the points of $\supp (\Gamma)$ and having no common component with $D_a$. 
Let $\Lambda$ be the $0$-cycle given by the scheme-theoretic intersection $D_a\cap D'$. Unfortunately, the support of $\Lambda$ might be bigger than that of $\Gamma$. Thus, 
we have $\Lambda=\Gamma+ \Gamma'$ where $\Gamma'$ is an effective $0$-cycle. Now choose a divisor $D''$ such that $D''\cap D_a =\Gamma'+\Gamma''$ where $\Gamma''$ and $\Gamma$ have disjoint supports. Set $D_b:=D'-D''$. The pair $(D_a,D_b)$ satisfies the criterion of proposition \ref{crit}, which concludes the proof.
\end{proof}

\begin{thm}
Let $G$ be a rational divisor on $S$, then for some canonical divisor $K$ and some divisor $D:=D_a+D_b$ such that $(D_a,D_b)$ is $\Delta$-convenient, we have
$$C_L(\Delta,G)=C_{\Omega}(\Delta,D_a,D_b, K-G+D).$$  
\end{thm}

\begin{proof}
Lemma \ref{exconv} asserts the existence of a $\Delta$-convenient pair $(D_a,D_b)$.
Then, construct a $2$-form $\omega_0$ using corollary \ref{omega0}.
Set $K:=(\omega_0)$.
Now the result is an easy consequence of theorem \ref{dif=fonc}.
\end{proof}


\section{The reverse inclusion is false}\label{examples}

As said in section \ref{suborth}, if a differential code is included in the orthogonal of a functional one, the reverse inclusion is in general false. The study of the following example will prove this.

In this section, the surface $S$ is the product of two projective lines $S:=\P^1 \times \P^1$.
Let $U$ be an affine chart of $S$ with affine coordinates $x,y$.
The complement of $U$ in $S$ is a union of two lines $E$ and $F$.
The Picard group of $S$ is generated by the classes of $E$ and $F$. 
Thus, without loss of generality, one can choose for $G$ the divisor $G_{n,m}:=mE+nF$, with $m,n\in \Z$.
Finally, $\Delta$ is defined as the formal sum of all rational points of $U$.

\subsection{Functional codes on $\mathbf{P^1\times P^1}$}
On $U$, the vector space $L(G_{m,n})$ may be identified with $\F_q[x]_{\leq m}\otimes_{\F_q}F_q[y]_{\leq n}$, where $\F_q[t]_{\leq d}$ denotes the space of polynomials in $t$ with degree less than or equal to $d$.
Furthermore, the functional code $C_L(\Delta, G_{m,n})$ may be identified with a tensor product of two codes on the projective line, which are Reed-Solomon codes. Thus,
\begin{equation}\label{tensRS}\tag{$\blacklozenge$}
C_L(\Delta, G_{m,n})=RS_q(m+1)\otimes_{\F_q} RS_q(n+1),
\end{equation} 
where $RS_q(k)$ denotes the Reed-Solomon code over $\F_q$ of length $q$ and dimension $k$.

\subsection{Orthogonal of functional codes on $\mathbf{P^1\times P^1}$}

In this subsection, we prove that the orthogonal of some functional code on $\P^1 \times \P^1$ cannot be differential.

\begin{prop}
  Let $m,n$ be two integers such that  $0\leq n,m < q-2$, then for all $\Delta$-convenient pair of divisors $(D_a,D_b)$, we have
$$C_{\Omega}(\Delta,D_a,D_b, G_{m,n})\varsubsetneq C_L(\Delta, G_{m,n})^{\bot}.$$
\end{prop}

\begin{proof}
From (\ref{tensRS}) and lemma \ref{orthotens} in appendix \ref{tens}, we have
$$
C_L(\Delta, G_{m,n})^{\bot}= RS_q(m+1)^{\bot} \otimes \F_q^q +\F_q^q \otimes
RS_q(n+1)^{\bot}.
$$
Suppose that for some $\Delta$-convenient pair $(D_a,D_b)$, we had
$$C_{\Omega}(\Delta,D_a,D_b, G_{m,n})=C_L(\Delta, G_{m,n})^{\bot}.$$

\noindent From theorem \ref{dif=fonc}, the code $C_{\Omega}(\Delta,D_a,D_b,G_{m,n})$ is functional. Thus, from (\ref{tensRS}), it is a tensor product of two Reed-Solomon codes.
But $C_L(\Delta, G_{m,n})^{\bot}$ is of the form $A\otimes \F_q^q +\F_q^q \otimes B$ with $A , B$ nonzero and strictly contained in $\F_q^q$.
This contradicts lemma \ref{lemtens} in appendix \ref{tens}.
\end{proof}

\begin{rem}
  The condition $0 \leq m,n < q-2$ asserts that in the tensor product representation $C_L(\Delta, G_{m,n})=RS_q(m+1) \otimes RS_q(n+1)$, none of the terms of the tensor product is zero or $\F_q^q$.
\end{rem}

A solution to avoid this lack of reverse inclusion is to try to construct the orthogonal as a sum of differential codes.
The purpose of next subsection is the realization of $C_L(\Delta,G_{m,n})^{\bot}$ as a sum of two differential codes associated with two distinct $\Delta$-convenient pairs.




\subsection{A construction of the orthogonal code}
For each $\alpha \in \F_q$, consider the lines $D_{1,\alpha}:=\{x=\alpha\}$, $D_{2, \alpha}:=\{y=\alpha\}$  $D_{3,\alpha}:=\{x-y-\alpha \}$. Now set
$$D_1:= \sum_{\alpha \in \F_q} D_{1, \alpha},\quad D_2:= \sum_{\alpha \in \F_q} D_{2,\alpha}
\quad and \quad D_3:=\sum_{\alpha \in \F_q}D_{3,\alpha}$$

\noindent Pairs $(D_1,D_3)$ and $(D_2,D_3)$ are $\Delta$-convenient. Using them, one can realize the orthogonal of $C_L(\Delta, G_{m,n})$ as a sum of two differential codes.

\begin{prop}
The three following relations are satisfied.
$$
\begin{array}{lccl}
(i) & C_{\Omega}(\Delta, D_1,D_3,G_{m,n}) & = & {\F}_q^q \otimes RS_q(q-2-n).  \\
(ii) & C_{\Omega}(\Delta,D_2,D_3,G_{m,n}) & = & RS_q(q-2-m) \otimes {\F}_q^q. \\
(iii) & {C_L(\Delta,G_{m,n})}^{\bot} & = & C_{\Omega}(\Delta,D_1,D_3 ,G_{m,n})+ C_{\Omega}(\Delta,D_2, D_3,G_{m,n}).
\end{array}
$$
\end{prop}

\begin{proof}
As said in last proof, relation (\ref{tensRS}) and lemma \ref{orthotens} entail
$$
C_L(\Delta, G_{m,n})^{\bot}= RS_q(m+1)^{\bot} \otimes \F_q^q +\F_q^q \otimes
RS_q(n+1)^{\bot}.
$$

\noindent Consequently, $(i)+(ii)\Rightarrow (iii)$. Furthermore, by symmetry $(i) \Leftrightarrow (ii)$. Thus, we only have prove $(i)$. Set
$$
\nu:= \frac{dy}{ \prod_{\beta \in \F_q}(x-\beta)} \w
\frac{ dx}{\prod_{\alpha \in \F_q}(x-y-\alpha)}.
$$
This form satisfies conditions (\ref{vvv}), (\ref{vvx}) and (\ref{vxx}) in corollary \ref{omega0}.
Compute the divisor of $\nu$. On $U$, we have $(\nu_{|U})=-{D_3}_{|U}-{D_1}_{|U}$, moreover
$D_1\sim qE$ and $D_3 \sim q(E+F)$.
Since the canonical class on $\P^1 \times \P^1$ equals that of $-2(E+F)$, we have
$$
(\nu)=(2q-2)E+(q-2)F-D_1-D_3
$$
and
$$
\begin{array}{rcl}
 C_{\Omega}(\Delta, D_1,D_3 ,mE+nF) & = & C_L(\Delta,(2q-2-m)E+(q-2-n)F)\\
 & = & RS_q(2q-2-m)\otimes RS_q(q-2-n).
\end{array}
$$
To conclude, notice that if $m\leq q-1$, then $2q-2-m\geq q-1$ and
$RS_q(2q-2-m)$ equals $\F_q^q$.
\end{proof}

\section{Conclusion}

This new construction of codes generalizes the differential construction of codes on curves.
The main difference is that it is not always possible to realize the orthogonal of a functional code as a differential (or equivalently functional) one. 
A natural question comes from the study of last example.

\begin{question}
Is the orthogonal of a functional code a sum of differential codes?
If yes, is there a bound on the number of differential codes involved in this sum?  
\end{question}

Moreover, we now know that the orthogonal of a functional code might be non-functional. Consequently, the study of such codes might be interesting.

\appendix
\section{Proof of lemma \ref{yval}}\label{apyval}
If we prove the well-definition of the morphism $ k((u))[[v]]\rightarrow k((x))[[y]]$, then we conclude about that of the morphism $k((u))((v))\rightarrow k((x))((y))$, using the universal property of fraction rings.
First, we have to define a topology on $k((u))[[v]]$ (resp. $k((x))[[y]]$). Recall that
$$
k((u))[[v]]=\lim_{\longleftarrow} \fract k((u))[v] / (v^n).
$$
Afterwards, using the $(u)$-adic topology of $k((u))$, one can define a topology of projective limit on $k((u))[[v]]$. For this topology, a sequence $(s^{(n)})_{n\in \N}$ defined for all $n$ by $s^{(n)}=\sum_{j\in\N} s_j^{(n)}(u)v^j$ converges to zero if and only if
$$
\forall j\in \N,\ \lim_{n\rightarrow +\infty} s_j^{(n)}(u)=0,\ \textrm{for the}
\ (u)\textrm{-adic topology}.
$$
Afterwards, using a Cauchy criterion, one proves that, for this topology, a series of elements of $k((u))[[v]]$ converges if and only if its general term converges to zero.

\begin{rem}\label{weaker}
Notice that this topology on $k((u))[[v]]$ is weaker than the $(v)$-adic one (for which the subset $k((u))$ is discrete). Thus, if a sequence (resp. a series) converges for the $(v)$-adic topology, hence converges for the projective limit topology.
\end{rem}

\begin{proof}[Proof of lemma \ref{yval}]
\textbf{Step 1.} Recall that $f$ is of the form $f=f_0(x)+f_1(x)v+\cdots$ and such that $f_0$ has $(x)$-adic valuation one. We will prove that the sequence $(f^n)_{n\in \N}$ converges to zero. Let $i$ be a nonnegative integer, for $n$ large enough, the coefficient of $y^i$ in $f^n$ is of the form
$
f_0^n P_i(f_0,\ldots,f_i)
$, where $P_i$ is a polynomial which does not depend on $n$. Thus, for the $(x)$-adic topology this coefficient tends to zero. Consequently, for all Laurent series $\phi(u)\in k((u))$, the series $\phi(f(x,y))$ converges in $k((x))[[y]]$.

\medbreak

\noindent \textbf{Step 2.} The series $g$ has $(y)$-adic valuation one, thus the sequence $(g^n)_{n\in \N}$ converges to zero for the $(y)$-adic topology, hence for the projective limit topology (see remark \ref{weaker}). Using step 1, we conclude that for every series $\psi(u,v) \in k((u))[[v]]$, the series $\psi(f,g)$ converges in $k((x))[[y]]$. Moreover, its $(y)$-adic valuation equals the $(v)$-adic one of $\psi$.

\medbreak

\noindent \textbf{Step 3.} If $\omega$ is a formal form $\omega=h(u,v) du\w dv$ with $(v)$-adic valuation $n\in \Z$, then we have to prove that the $(y)$-adic valuation of $h(f,g)df\w dg$ is $n$ too.
If we prove that the $(y)$-adic valuation of $df\w dg$ is zero, then we can conclude using step 2. For that, consider the expression
$$
df\w dg= \Bigg( \underbrace{\frac{\p f}{\p x} \frac{\p g}{\p y}}_{val_{(y)}=0} -
\underbrace{ \frac{\p f}{\p y} \frac{\p g}{\p x}}_{val_{(y)}\geq 1} \Bigg) dx\w dy.
$$
This concludes the proof.

\end{proof}



\section{Proof of lemma \ref{Jacob}}\label{apJacob}
If $\omega=dA\w dB$ for some series $A,B\in k((u))((v))$, after a change of coordinates (\ref{cv}), $\omega=dF\w dG$
for some other series $F,G\in k((x))((y))$.
Then, in order to prove the lemma we only have to prove that the $(u,v)$-$2$-residue of $\omega=dA\w dB$ is zero.

We first introduce some notations.
Let $\rho$ and $\textrm{Jac}$ be the maps  
$$\rho :\left\{\begin{array}{ccc}
k((x))((y)) & \rightarrow & k((x)) \\
\sum_{i\geq -n} h_j(u)v^j & \mapsto & h_{-1}(u)
\end{array},\right.
\ 
\textrm{Jac} :\left\{\begin{array}{ccc}
k((x))((y))^2 & \rightarrow & k((x))((y)) \\
(A,B) & \mapsto & \frac{\partial A}{\partial u} \frac{\partial
  B}{\partial v} - \frac{\partial A}{\partial v} \frac{\partial B}{\partial u}
\end{array}.\right.
$$
Thus, $\omega=dA\w dB=\textrm{Jac}(A,B)du\w dv$. We will prove the following lemma.

\begin{lem}\label{mescouilles}
For all $A,B\in k((u))((v))$, we have $\rho(\textrm{Jac}(A,B))=\phi'(u)$ for some $\phi \in k((u))$, where $\phi'$ denotes the formal derivative of $\phi$. 
\end{lem}

\begin{proof}
Maps $\textrm{Jac}$ and $\rho$ are respectively $k$-bilinear antisymmetric and $k$-linear.
Then, we can restrict the proof to the three following situations and extend it by linearity.
\begin{enumerate}
\item\label{sit1} $A, B \in k((u))[[v]]$.
\item\label{sit2} $A\in k((u))[[v]]$ and $B=\frac{b(u)}{v^{n}}$
  with $n\in \mathbf{N}^*$ and $b\in k((u))$.
\item\label{sit3} $A=\frac{a(u)}{v^{m}}$ and $B=\frac{b(u)}{v^{n}}$ with $m,n \in
  \mathbf{N}^*$ and $a,b\in k((u))$.
\end{enumerate}

\noindent Let us consider these three situations.

\noindent (\ref{sit1}) The series $A$ and $B$ don't have terms with negative powers of $v$, thus so are their partial derivatives, then $\rho \circ \textrm{Jac}(A,B)=0$.

\noindent (\ref{sit2}) The series $A$ is of the form $A=\sum_{j\geq 0}a_j(u)v^j$. Then, 
$$
\rho (\textrm{Jac} (A,B)) = -n(a_n'(u)b(u)+a_n(u)b'(u))=(-na_n(u)b(u))'.
$$

\noindent (\ref{sit3})
$$\textrm{Jac} (A,B)=\left(-n \frac{a' (u)b (u)}{v^{m+n+1}} - (-m)\frac{a (u)b'(u)}{v^{m+n+1}} \right).$$
Integers $m$ and $n$ are positives, there is no term in $v^{-1}$, thus $\rho (\textrm{Jac} (A,B))=0$.
\end{proof}

\noindent \textbf{Conclusion.}
Using lemma \ref{mescouilles}, we get $(u,v)\res^1_{C,P}(\omega)=\phi'(\bar{u})d\bar{u}$ and the coefficient of $u^{-1}$ in $\phi'$ is zero, because it is a derivative.  

\section{Proof of theorem \ref{inv2res} when $k$ has a positive characteristic}\label{apinv2res}

We only have to work on the points of the proof of proposition \ref{inv2res} in which we used specific properties of characteristic zero.
Thus, we will study the behavior under (CV2) of differentials of the form
$$\omega=\phi(u)du\w \frac{dy}{y^{n+1}},\quad \textrm{where}\quad \phi \in k((u))\quad \textrm{and}\quad n\geq 1.$$
Let $N$ be a nonnegative integer.
In what follows, we consider a change of variables of the form (CV2)
$$u=f(x,y),\quad \textrm{with}\quad f=\sum_{j\geq 0}f_j(x)y^j\quad \textrm{and}\quad f_0 \in xk[[x]]\smallsetminus x^2k[[x]],$$
such that $\min_{k=1\ldots n} \left\{val_{(x)}(f_k)\right\}=-N$, where $val_{(x)}$ denotes the $(x)$-adic valuation of an element of $k((x))$.

\medbreak

\noindent \textbf{Step 1.}
Assume that $\omega=u^mdu\w \frac{dy}{y^{n+1}}$ with $m\in \N$.
Then,
$$
\omega = (f'_0(x)+f'_1(x)y+\cdots)(f_0(x)+f_1(x)y+\cdots)^m dx \w \frac{dy}{y^{n+1}}.
$$

\noindent The $(x,y)$-$1$-residue of $\omega$ is the coefficient in $y^{n}$ of the series $f^m \p f/\p x$. This residue is of the form
$$
(x,y)\res^1_{C,P}(\omega)=P_{m,n}(f_0,\ldots,f_n,f'_0,\ldots,f'_n)d\bar{x},
$$
where $P_{m,n}\in \Z[X_0,\ldots,X_n,Y_0,\ldots, Y_n]$  depends neither on the field $k$ nor on $f$. Actually, $P_{m,n}$ depends only on $m$ and $n$. By the same way, its coefficient of $x^{-1}$ is a polynomial expression $Q$ in the $f_{i,j}$'s with $0 \leq j \leq n$ and $-N \leq i \leq N+1$, such that $Q$ has coefficients in $\Z$ and depends neither on $k$ nor on $f$. Furthermore, if $k$ has characteristic zero, we know from section \ref{general} that $Q$ vanishes on the set $\{f_{1,0}\neq 0\}$, hence is the zero polynomial.

\medbreak

\noindent \textbf{Step 2.}
Assume that $\omega=\phi(u)du\w \frac{dy}{y^{n+1}}$, with $\phi=\sum_{m\geq 0} \phi_m u^m \in k[[u]]$. From step 1, we have
\begin{equation}\label{seriegolo}
(x,y)\res^1_{C,P}(\omega)=\sum_{m\geq 0} \phi_m P_{m,n}(f_0,\ldots, f_n,f'_0,\ldots,f'_n)d\bar{x},
\end{equation} 
where $P_{m,n}$'s denote the polynomials involved in Step 1.
The $(x,y)$-$1$-residue of $\omega$ is well-defined. Thus, the series in (\ref{seriegolo}) converges in $k((x))$. Consequently, the $(x)$-adic valuation of it terms is positive for each $m\geq M$ and
$$
(x,y)\res^1_{C,P}(\omega)=\sum_{m=0}^M \phi_m P_{m,n} d\bar{x} + \sum_{m>M} \phi_m P_{m,n} d\bar{x}.
$$ 

\noindent The right term has positive $(x)$-adic valuation, thus its residue is zero. The left one has zero residue zero because of step 1 extended by linearity.

\medbreak

\noindent \textbf{Step 3.}
Assume that $\omega=\frac{du}{u^m} \w \frac{dy}{y^{n+1}}$, with $m \in \N$.
Then, $
\omega  = \frac{1}{f^m} \frac{\p f}{\p x}  dx \w \frac{dy}{y^{n+1}}
$.
We have to study the fraction $\frac{1}{f^m} \frac{\p f}{\p x}$.
First, compute its coefficient of $y^n$ corresponding to the $(x,y)$-$1$-residue. We have
\begin{equation}\label{RR}
 \frac{1}{f^m}= \frac{1}{f_0^m}
\left(1+ \frac{R_{m,1}(f_0,f_1)}{f_0}y+\cdots 
+\frac{R_{m,p}(f_0,\ldots,f_p)}{f_0^p}+\cdots \right),
\end{equation}
for some homogeneous polynomials $R_{m,i} \in \Z[X_0,\ldots,X_p]$ of degree $p$ and depending only on $m$. Thus, the coefficient of $y^n$ in $\frac{1}{f^m}\frac{\p f}{\p x}$ is
$$
C(x):= \frac{1}{f_0^{m}}\left(f'_n+ f'_{n-1}\frac{R_{m,1}(f_0,f_1)}{f_0}+\ldots+f'_0 \frac{R_{m,n}(f_0,\ldots,f_n)}{f_0^n} \right).
$$
For all $k \in \{1,\ldots,n\}$, set $S_{m,n,k}(f_0,\ldots,f_k):=f_0^{n-k}R_{m,k}(f_0,\ldots,f_k)$ and $S_{m,n,0}(f_0):=f_0^n$. Polynomials $S_{m,n,k}$'s are homogeneous of degree $n$ and
\begin{equation}\label{Cx}
C(x):= \underbrace{\frac{1}{f_0^{m+n}}}_{A(x)}\  \underbrace{\sum_{k=0}^n f_{n-k}' S_{m,n,k}(f_0,\ldots,f_k)}_{B(x)}.
\end{equation}
Recall that $f_0\in xk[[x]]$, that is $f_0:=f_{1,0}x+f_{2,0}x^2+\cdots$, then
$$
A(x)=\frac{1}{(f_{1,0}x)^{m+n}}\Bigg(1+\frac{R_{m,1}(f_{1,0},f_{2,0})}{f_{1,0}}x+\cdots \qquad \qquad \qquad  \qquad  \qquad 
$$

$$
\qquad \qquad \qquad \qquad  \qquad  \cdots +\frac{R_{m,p-1}(f_{1,0},f_{2,0},\ldots ,f_{p,0})}{f_{1,0}^p}x^p+\cdots \Bigg),
$$
for polynomials $R_{m,k}$ as in (\ref{RR}).
Finally, we want to express the coefficient $C_{-1}$ of $x^{-1}$ in $C(x)$.
Recall that, for all $p$, the degree of $S_{m,n,p}$ is $n$.
Therefore, there exists an integer $M$ and a polynomial $V\in \Z[X_{i,j}]$ with $-N \leq i \leq \max \left(m+n\ ,\ (n+1)N+1\right)$ and
$ \ 0 \leq j \leq n$, depending only on $m$, $n$ and $N$ and such that
$$C_{-1}=\frac{1}{f_{1,0}^M}V(f_{i,j}).$$

\noindent Over a field of characteristic zero, $V$ vanishes on the set $\{f_{1,0}\neq 0\}$, hence is the zero polynomial.

\begin{rem}
Notice that, in the whole proof, we deal with the value $N$ such that $-N$ is the minimal valuation of the $f_i$, for $0\leq i\leq n$.
Then, we have proved that the $2$-residue is invariant under a change of variables $u=f(x,y)$ such that the $f_i$ have valuation minored by $-N$.
But we proved it for all $N$, which concludes the proof.   
\end{rem}

\section{About tensor products}\label{tens}

Statements of this appendix are quite elementary results of linear algebra. We prove them because of a lack of references.

\begin{lem}\label{orthotens}
  Let $(E,<,>_E)$ and $(F,<,>_F)$ be two finite-dimensional vector spaces over an arbitrary field $k$ with respective non-degenerate bilinear forms $<\ \!,>_E$ and $<\ \!,>_F$. Let $A$ and $B$ be respective subspaces of $E$ and $F$, then for the bilinear form $<\ \!,>_{E \otimes F}:=<\ \!,>_E \otimes <\ \!,>_F$ on $E\otimes_k F$, we have
$$
(A\otimes_k B)^{\bot}=A^{\bot} \otimes_k F + E \otimes_k B^{\bot}.
$$
\end{lem}

\begin{proof}
  Inclusion ``$\supseteq$'' is obvious. For the reverse one, we will prove that both spaces have the same dimension.
First, we have to prove that
\begin{equation}\label{equalspaces}
A^{\bot} \otimes_k F \cap E \otimes_k B^{\bot}=A^{\bot} \otimes B^{\bot}.
\end{equation}

\noindent Here again, inclusion ``$\supseteq$'' is obvious. For the reverse one, consider bases $(e_i)_{i\in I_0}$ and $(f_j)_{j\in J_0}$ respectively of $A^{\bot}$ and $B^{\bot}$ and complete them as bases $(e_i)_{i\in I}$ and $(f_j)_{j\in J}$ of $E$ and $F$. Then, for all $s=\sum_{i,j}s_{ij}e_i\otimes f_j \in E\otimes F$, we have
$$
\ s\in A^{\bot} \otimes_k F \cap E \otimes_k B^{\bot} \Longrightarrow \Big( \forall (i,j) \in (I\smallsetminus I_0)\times (J \smallsetminus J_0), \ s_{ij}=0 \Big).
$$
Thus, (\ref{equalspaces}) is proved and entails
$$
\dim (A^{\bot} \otimes F + E \otimes B^{\bot})=
\dim (A^{\bot} \otimes F) + \dim (E \otimes B^{\bot}) -
\dim (A^{\bot} \otimes B^{\bot}).
$$

\noindent After an easy computation, we prove that spaces $(A \otimes B)^{\bot}$ and $(A^{\bot}\otimes F+E \otimes B^{\bot})$ have the same dimension, which concludes the proof.
\end{proof}

\begin{lem}\label{lemtens}
Let $E$ and $F$ be two vector spaces over an arbitrary field $k$. Let $A$ (resp. $B$) be a strict nonzero subspace of $E$ (resp $F$). Then, the subspace $A\otimes_k F + E\otimes_k B$ of $E\otimes_k F$ cannot be written as an elementary tensor product $U\otimes V$.  
\end{lem}

\begin{proof}
Assume that $A\otimes F + E\otimes B =U\otimes V$ for some subspace $U$ (resp. $V$) of $E$ (resp. $F$).

\medbreak

\noindent Let $(e_i)_{i \in I_0}$ (resp. $(f_j)_{j \in J_0}$) be a basis of $A$ (resp. $B$) completed in a basis $(e_i)_{i \in I}$ of $E$ (resp. $(f_j)_{j \in J}$ of $F$).
Assume that $U \nsubseteq A$ and choose $u \in U$ such that $u \notin A$. Then, for all $v \in V$, the vector $u \otimes v$ is of the form
$$
u \otimes v=\sum_{i,j}u_i v_j e_i \otimes f_j.
$$
From the assumption $A\otimes F + E\otimes B =U\otimes V$, the product $u_iv_j$ is zero for all couple $(i,j) \in I\smallsetminus I_0 \times J\smallsetminus J_0$. Since $u \notin A$, there exists at least one index $i_1\in I\smallsetminus I_0$ such that $u_{i_1}\neq 0$. Thus, for all $j\in J\smallsetminus J_0$, we have $u_{i_1}v_j=0$ which entails that $v\in B$. 
This statement works for all $v\in V$, hence $U\otimes V \subseteq E\otimes B$. Now choose $f\in F$ such that $f\notin B$ and $a\in A\smallsetminus \{0\}$.
Then, $a \otimes f \notin E\otimes B$, thus $a \otimes f \notin U\otimes V$ which contradicts $A\otimes F + E\otimes B =U\otimes V$.

\medbreak

\noindent If $U\subseteq A$, use the same argument replacing $U,A,E$ by $V,B,F$.
\end{proof}


\section*{Acknowledgments}
I would like to thank Gerhard Frey for his explanations about valuations, Emmanuel Hallouin for his fruitful advices in commutative and local algebra and the anonymous referee for his conscientious work. Moreover, I'm also very grateful to my advisor Marc Perret, for all his so constructive comments about this paper. 

\bibliographystyle{abbrv}
\bibliography{resbibi}

\end{document}